\documentclass[12pt,draftcls,onecolumn]{IEEEtran}

\usepackage[ansinew]{inputenc}
\usepackage{amsmath, amsfonts} 
\usepackage{amsthm, amssymb}
\usepackage[english]{babel}
\usepackage{graphicx}
\usepackage{subfigure}

\newtheorem{theorem}{Theorem}
\newtheorem{proposition}{Proposition}

\newtheorem{corollary}{Corollary}

\newtheorem{property}{Property}

\theoremstyle{definition}
\newtheorem{definition}{Definition}
\newtheorem{remark}{Remark}
\newtheorem{example}{Example}
\usepackage{verbatim}
\usepackage{array}
\usepackage{float}
\usepackage{subfigure}
\usepackage{color}

\newcommand{\blue}[1]{\textcolor{black}{#1}}

\begin{document}

\title{Global stability analysis using the eigenfunctions of the Koopman operator\footnote{Part of this work was funded by Army Research Office Grant W911NF-11-1-0511 and was performed while A. Mauroy was with the Department of Mechanical Engineering, University of California Santa Barbara. A. Mauroy is currently supported by a BELSPO (Belgian Science Policy) return grant.}}
\author{Alexandre Mauroy and Igor Mezi\'c}

\maketitle
\begin{abstract}

We propose a novel operator-theoretic framework to study global stability of nonlinear systems. Based on the spectral properties of the so-called Koopman operator, our approach can be regarded as a natural extension of classic linear stability analysis to nonlinear systems. The main results establish the (necessary and sufficient) relationship between the existence of specific eigenfunctions of the Koopman operator and the global stability property of \blue{hyperbolic} fixed points and limit cycles. These results are complemented with numerical methods which are used to estimate the region of attraction of the fixed point or to prove in a systematic way global stability of the attractor within a given region of the state space.

\end{abstract}
\thanks{A. Mauroy is with the Department of Electrical Engineering and Computer Science, University of Liège, Belgium. I. Mezi\'c is with the Department of Mechanical Engineering, University of California Santa Barbara, Santa Barbara, CA 93106, USA. E-mails: a.mauroy@ulg.ac.be; mezic@engineering.ucsb.edu.}\\

\section{Introduction}

\blue{
Dynamical systems are traditionally described with their trajectories evolving in a finite-dimensional state space. In contrast to this pointwise description, there exists an alternative operator-theoretic description: the systems can be represented by an operator acting on an infinite-dimensional space of functions. For instance, the so-called Koopman operator describes the evolution of observable functions along the trajectories (see e.g. \cite{Budisic_Koopman,Mezic_ann_rev} for a review). One iteration of the Koopman operator acting on one observable is equivalent to an iteration along all the trajectories of the system. Hence the operator-theoretic description provides a global insight into the system dynamics which is appropriate for global stability analysis.
}

\blue{The operator-theoretic approach to nonlinear systems is used implicitly in classical methods for global stability analysis. In particular, Lyapunov's second method does not rely on the pointwise description of the system but rather on the operator-theoretic framework.} A Lyapunov function is indeed a particular observable function that decreases through the action of the Koopman operator. In addition, the counterpart method based on densities---proposed only a few years ago in \cite{Rantzer_density}---is directly related to the dual operator (i.e. the Perron-Frobenius operator \cite{Lasota_book})---which is known for decades. In this context, it is surprising that explicit operator-theoretic techniques for stability analysis have only been developed recently \cite{Umesh_density,Umesh}. Equivalent optimization methods on occupation measures (i.e. in the dual Perron-Frobenius framework) for region of attraction estimation are also very recent \cite{Henrion2}.

The Koopman operator theory provides a well-developed framework for the study of dynamical systems. Even when associated with nonlinear dynamics, the Koopman operator is linear (but infinite dimensional) and thus amenable to a systematic spectral analysis which reveals the system behavior. More precisely, its spectral properties are strongly connected to the geometric properties of the system dynamics: the Koopman eigenfunctions capture periodic partitions in ergodic systems \cite{Mezic}, isochrons \cite{Mauroy_Mezic} and isostables \cite{MMM_isostables} in dissipative systems, and are related to the global linearization of the system \cite{Lan}.
Moreover, this spectral approach is conducive to data analysis and has for instance been used recently to predict instabilities in power grids \cite{Susuki}.
However, in spite of its success, a theoretical framework based on the spectral properties of the Koopman operator has received so far little attention in the context of global stability analysis.

Building on preliminary results presented in \cite{MauroyMezic_CDC}, this paper investigates the interplay between the global stability properties of a nonlinear system and the spectral properties of the associated Koopman operator. \blue{The proposed framework is very general, but we mainly focus on the global stability analysis of hyperbolic attractors, obtaining in this case necessary and sufficient conditions which rely on the existence of specific Koopman eigenfunctions.} These results mirror classic stability results of linear systems (or equivalently local stability results of nonlinear systems). While our previous work \cite{MauroyMezic_CDC} focuses on the case of a \blue{hyperbolic} fixed point, we also consider the case of a \blue{hyeprbolic} limit cycle, exploiting the fact that the proposed operator-theoretic approach is general and does not require specific assumptions on the nature of the attractor.

New numerical methods are also developed for the computation of the Koopman eigenfunctions. In contrast to existing methods (e.g. Fourier/Laplace averages \cite{MMM_isostables,Mezic,Mohr}, Arnoldi-type methods \cite{Rowley}), the numerical schemes proposed in this paper do not require the integration of (a finite number of) particular trajectories, so that they can be used as systematic methods for global stability analysis and estimation of the basin of attraction.  These numerical techniques are based on the decomposition of the eigenfunctions on a polynomial basis. Complementing a previous method using Taylor polynomials \cite{MauroyMezic_CDC}, we propose a novel method based on Bernstein polynomials which can be used when the eigenfunctions are not analytic.

The paper is organized as follows. In Section \ref{sec_Koopman_intro}, the operator-theoretic framework is introduced and general stability results using the Koopman operator are provided. Section \ref{sec_Eigenfunctions} investigates the relationship between the spectral properties of the Koopman operator and the global stability properties of the system. \blue{A general stability result is given and then applied to hyperbolic fixed points and limit cycles}. In Section \ref{sec_num_methods}, numerical methods using either Taylor polynomials or Bernstein polynomials are developed and illustrated with several examples. Concluding remarks are given in Section \ref{conclu}.

\section{Preliminary stability results with the Koopman operator}
\label{sec_Koopman_intro}

This section introduces the operator-theoretic framework and, as a preliminary to the main results, presents general stability results related to the asymptotic behavior of the Koopman operator. \blue{Most of the results and concepts can also be found in \cite{MauroyMezic_CDC}.}
 
\subsection{Operator-theoretic approach to dynamical systems}

We consider a flow $\varphi(t,x)$ on an arbitrary set $X$---i.e. $\varphi:\mathbb{R} \times X \rightarrow X$ satisfies the (semi)group properties $\varphi(0,x)=x$ and $\varphi(s,\varphi(t,x))=\varphi(s+t,x)$---and a (Banach) space $\mathcal{F}$ of \emph{observables} $f:X \rightarrow \mathbb{C}$. We assume that the observables are continuous, i.e. $\mathcal{F} \subseteq C^0(X)$. The so-called Koopman operator associated with $\varphi$ is defined on $\mathcal{F}$ as follows.
\begin{definition}[Koopman operator]
The Koopman (semi)group of operators $U^t:\mathcal{F} \rightarrow \mathcal{F}$ associated with the flow $\varphi$ is defined by
\begin{equation}
\label{Koopman_op}
U^t f = f \circ \varphi^t \qquad f\in \mathcal{F} \,.
\end{equation}
\end{definition}
No specific assumption on the flow $\varphi$ or on the set $X$ is needed to define the Koopman operator. For instance, $\varphi$ might be induced by a well-defined hybrid system. In the following, the set $X \subset \mathbb{R}^N$ is assumed to be compact and forward invariant under $\varphi^t(\cdot)\triangleq\varphi(t,\cdot)$ (i.e. $\varphi^t(X) \subseteq X$ $\forall t\geq 0$). In addition, we will consider that the flow is induced by 
the dynamical system
\begin{equation}
\label{syst}
\dot{x}=F(x) \,, \quad x\in  \mathbb{R}^N \,,
\end{equation}
i.e. $\varphi^t(x_0)$ is the solution of \eqref{syst} associated with the initial condition $x_0 \in \mathbb{R}^N$. If in addition $f$ and $F$ are continuously differentiable, $g(t,x)=U^t f(x)$ is the solution of the partial differential equation (see e.g. \cite{Lasota_book})
\begin{equation}
\label{PDE_Koopman}
\frac{\partial g}{\partial t}=F \cdot \nabla g \triangleq L_U g
\end{equation}
with the initial condition $g(0,x)=f(x)$. The symbol $\nabla$ denotes the gradient and $\cdot$ is the inner product in $\mathbb{R}^N$. The operator $L_U$ is the infinitesimal generator of $U^t$, i.e. $L_U f = \lim_{t\rightarrow 0} (U^t f-f)/t$. No boundary condition is added to \eqref{PDE_Koopman} since $X$ is forward invariant.

The (infinite-dimensional) Koopman operator representation \eqref{Koopman_op} is equivalent to the (finite-dimensional) system representation \eqref{syst}. But a remarkable fact is that $U^t (a f_1+ b f_2)=a U^t f_1 + b U^t f_2$, with $f_1,f_2 \in \mathcal{F}$ and $a,b \in \mathbb{R}$, so that the Koopman operator is \emph{linear} even when \eqref{syst} is nonlinear.\\

\paragraph*{Duality and Perron-Frobenius operator}
According to the theory of linear operators on Banach spaces, the Koopman operator $U^t$ has a dual operator $P^t$ which acts on the conjugate space $\mathcal{F}'$ of bounded linear functionals $\psi$ on $\mathcal{F}$, according to the relationship $P^t \psi=\psi \circ U^t$. Since $\mathcal{F}\subseteq C^0(X)$, each bounded linear functional $\psi \in\mathcal{F}'$ can be associated with a Radon measure and the operator $P^t$ equivalently acts on a space of measures. When the linear bounded functionals (or equivalently the measures) can be associated with a density $\rho \in \mathcal{F}^\dagger:X \rightarrow \mathbb{C}$ according to
\begin{equation*}
\psi(f)=\langle f, \rho \rangle_\mu = \int_X f \, \rho \, \mu(dx)  \qquad \forall f \in \mathcal{F}\,,
\end{equation*}
where $\mu$ on $X$ is a given finite measure, the dual operator $P^t$ can be redefined on $\mathcal{F}^\dagger$ and satisfies $\langle U^t f, \rho \rangle_\mu = \langle f, P^t \rho \rangle_\mu$. In this case, the dual operator $P^t$ is the so-called Perron-Frobenius operator describing the transport of densities along the trajectories of the flow $\varphi$.

If the flow is induced by \eqref{syst} and if $\rho$ and $F$ are continuously differentiable, the evolution of the density $g(t,x)=P^t \rho(x)$ satisfies the transport equation (see e.g. \cite{Lasota_book})
\begin{equation}
\label{PDE_PFO}
\frac{\partial g}{\partial t}=-\nabla \cdot (F g) \triangleq L_{P} g
\end{equation}
with the initial condition $g(0,x)=\rho(x)$. In contrast to the case of Koopman operator, proper boundary conditions must be considered with \eqref{PDE_PFO}.

\subsection{First stability results}
\label{sec_first_results}

The Koopman operator and the Perron-Frobenius operator provide two equivalent descriptions of the system (i.e. point-wise and set-wise descriptions) which can be related to well-known notions of stability theory. For the system \eqref{syst} admitting a globally stable fixed point $x^* \in X$, a \emph{Lyapunov function} $\mathcal{V}$ can be regarded as a nonnegative observable that decreases through the action of the Koopman operator, i.e. $L^t_U \mathcal{V}(x)<0$ for all $x \neq x^*$. Similarly, the more recent notion of \emph{Lyapunov density} introduced in \cite{Rantzer_density} is a function $C^1(X\setminus \{x^*\})$ that satisfies $\nabla \cdot (F \rho)>0$. This precisely corresponds to a density that decreases under the action of the Perron-Frobenius operator, i.e. $L^t_P \rho(x)<0$ for all $x \neq x^*$ \cite{Umesh_density,Umesh}.

The Perron-Frobenius approach---related to the Lyapunov density---is of particular interest to capture the weaker notion of almost everywhere stability, instead of the classical notion of stability. However, the spectral methods that we develop in the next section are more suited to the Koopman operator framework (see Remark \ref{rem_eigenfunction} below), which we will exclusively consider in the rest of the paper. The reader may wish to refer to \cite{Umesh_density,Umesh} for stability results obtained with the Perron-Frobenius operator.\\

\paragraph*{Definitions} At this point on we assume that the flow $\varphi$ admits an attractor $A$ with global stability properties on $X$.
\begin{definition}[Attractor]
\label{def_attractor}
The set $A\subset X$ is an attractor of the flow $\varphi^t$ if it satisfies all of the following properties:
\begin{enumerate}
\item $A$ is forward invariant under $\varphi^t(\cdot)$;
\item There exists a neighborhood $V \subset X$ of $A$ such that the limit set $\omega(x)$ of every $x\in V$ is in $A$, i.e.
\begin{equation*}
\omega(x) \triangleq \bigcap_{T \in \mathbb{R}} \overline{\{\varphi^t(x),t>T\}} \subseteq A  \qquad \forall x \in V\,,
\end{equation*}
where $\overline{\phantom{A}}$ denotes the closure of the set;
\item There exists no strictly smaller closed subset satisfying the above properties.
\end{enumerate}
\end{definition}
\blue{
\begin{definition}[Global stability]
\label{def_stability} An attractor (or a set) $A$ is
\begin{itemize}
\item stable if, for every $\epsilon>0$, there exists $\delta>0$ such that
\begin{equation*}
\min_{y\in A} \|x-y\|<\delta\,, x\in X \quad \Rightarrow \quad \min_{y\in A} \|\varphi^t(x)-y\|<\epsilon \quad \forall t \geq 0\,;
\end{equation*} 
\item globally attractive in $X$ if $\omega(x) \subseteq A$ for all $x\in X$;
\item globally asymptotically stable if it is stable and globally attractive.
\end{itemize}
\end{definition}
}

\paragraph*{Decomposition of the Koopman operator} As a preliminary to the main results of this section, we introduce the subspace $\mathcal{F}_{A_c} \subseteq \mathcal{F}$ of functions with support on $A_c=X \setminus A$, i.e.
\begin{equation*}
\mathcal{F}_{A_c}=\{f \in \mathcal{F} | f(x)=0 \, \forall x\in A\}\,.
\end{equation*}
If $x\in A$, one has $\varphi^t(x)\in A$ for all $t \in \mathbb{R}$, so that $U^t f(x) = f \circ \varphi^t(x)=0$ for $f\in\mathcal{F}_{A_c}$ and $x\in A$. It follows that $\mathcal{F}_{A_c}$ is invariant under $U^t$, and we denote by $U_{A_c}^t$ the restriction of $U^t$ to $\mathcal{F}_{A_c}$. Similarly, we consider the set $\mathcal{F}|_A$ of observables restricted to $A$, i.e. $\mathcal{F}|_A=\{f|_A:A\rightarrow \mathbb{C}| f\in \mathcal{F}\}$. Note that the extension 
$\bar{f}$ of $f \in \mathcal{F}|_A$ to $X$, such that $\bar{f}=f$ on $A$ and $\bar{f}=0$ on $A_c$, is generally not in $\mathcal{F} \subseteq C^0(X)$. The Koopman operator $U_{A}^t:\mathcal{F}|_A \rightarrow \mathcal{F}|_A$ is associated with the flow on $A$, i.e. $\varphi|_A:\mathbb{R}\times A \rightarrow A$, and is rigorously defined by $U_A^t(f|_A)=f|_A \circ \varphi^t|_A=(U^t f)|_A$.\\

\paragraph*{General stability results} Stability properties of the attractor are captured by the restriction $U_{A_c}^t$ of the Koopman operator. The result is summarized in the following proposition \blue{(see also \cite{MauroyMezic_CDC})}.
\begin{proposition}
\label{prop_Koopman_stab}
The attractor $A$ of \eqref{syst} is globally \blue{attractive} in $X$ if and only if
\begin{equation}
\label{stab_cond_U}
\lim_{t \rightarrow \infty} U_{A_c}^t f = 0 \quad \forall f\in \mathcal{F}_{A_c}\,,
\end{equation}
with $\mathcal{F} = C^0(X)$.
\end{proposition}
\begin{proof}
\emph{Sufficiency.} Consider the distance function 
\begin{equation*}
d(x)=\min_{y\in A} \|x-y\|\,, \quad x\in X\,.
\end{equation*}
The function is continuous and has zero value on $A$. Then, it follows from \eqref{stab_cond_U} that $\lim_{t \rightarrow \infty} d(\varphi^t(x))=0$. This implies that the limit set $\omega(x)\subseteq A$ for all $x\in X$, so that $A$ is globally \blue{attractive}.\\
\emph{Necessity.} For $f\in\mathcal{F}_{A_c}$, we have
\begin{equation*}
\lim_{t\rightarrow \infty} U_{A_c}^t f (x)= \lim_{t\rightarrow \infty} f \circ \varphi^t(x) = f \circ \lim_{t\rightarrow \infty} \varphi^t(x)\,,
\end{equation*}
since $f$ is continuous. Global \blue{attractivity} implies that $\omega(x) \subseteq A$ for all $x\in X$ and therefore $\lim_{t\rightarrow \infty} U_{A_c}^t f (x)=0$ for all $x\in X$.
\end{proof}

The result of Proposition \ref{prop_Koopman_stab} is very general, since it makes no assumption on the type of attractor $A$, which can even be a set of several attractors. In addition, the result holds not only for flows induced by the dynamics \eqref{syst} but also for any well-defined flow. An equivalent result can also be obtained for exponential stability (see \cite{MauroyMezic_CDC}).

\section{Koopman eigenfunctions and global stability}
\label{sec_Eigenfunctions}

The results presented in Section \ref{sec_first_results} involve all the observables of the chosen functional space and are therefore difficult to use in practice. In this section, we show that only a few particular functions---i.e. the eigenfunctions of the Koopman operator---are sufficient to capture the stability properties of the system.

\subsection{Properties of the Koopman eigenfunctions}

\blue{We define the eigenfunctions of the Koopman operator and summarize their main properties (see also \cite{Mezic,Mezic_ann_rev,MauroyMezic_CDC}).}

\begin{definition}[Koopman eigenfunction]
An eigenfunction of the Koopman operator (or in short, a Koopman eigenfunction) is an observable $\phi_\lambda\in\mathcal{F} \neq 0$ that satisfies
\begin{equation}
\label{eig_evol}
U^t \phi_\lambda = e^{\lambda t} \phi_\lambda
\end{equation}
for some $\lambda\in \mathbb{C}$. The value $\lambda$ is the associated Koopman eigenvalue and belongs to the point spectrum of the operator.
\end{definition}
If $F\in C^1(X)$, it follows from \eqref{PDE_Koopman} that the Koopman eigenfunctions satisfy the eigenvalue equation
\begin{equation}
\label{eig_equa}
L_U \phi_\lambda = F \cdot \nabla \phi_\lambda = \lambda \phi_\lambda \,.
\end{equation}
\begin{remark}
\label{rem_eigenfunction}
Koopman eigenfunctions are smooth in the vicinity of the attractor, a property which contrasts with the case of the dual Perron-Frobenius operator. Indeed, the eigenfunctions of the Perron-Frobenius operator are Dirac functions (or $n$th derivative of Dirac functions) with support on the attractor and it can be shown that each of these only captures local information of the dynamics. In the context of global stability analysis, it is therefore much more appropriate to consider the eigenfunctions of the Koopman operator.
\end{remark}

Koopman eigenfunctions and eigenvalues are characterized by the following property.
\begin{property}
\label{property_eig}
Suppose that $\phi_{\lambda_1}$ and $\phi_{\lambda_2}$ are two Koopman eigenfunctions associated with the eigenvalues $\lambda_1$ and $\lambda_2$. If $\phi_{\lambda_1}^{k_1} \, \phi_{\lambda_2}^{k_2} \in \mathcal{F}$, with $k_1,k_2 \in \mathbb{R}$, then it is an eigenfunction associated with the eigenvalue $k_1 \lambda_1+k_2 \lambda_2$.\\
\end{property}
\begin{proof}
It follows from \eqref{eig_evol} that
\begin{equation*}
U^t(\phi_{\lambda_1}^{k_1}\, \phi_{\lambda_2}^{k_2})= U^t \phi_{\lambda_1}^{k_1} \, U^t \phi_{\lambda_2}^{k_2} = (U^t \phi_{\lambda_1})^{k_1} \, (U^t \phi_{\lambda_2})^{k_2} = e^{(k_1 \lambda_1+k_2 \lambda_2) t} \phi_{\lambda_1}^{k_1} \, \phi_{\lambda_2}^{k_2} \,.
\end{equation*}
\end{proof}
Property \ref{property_eig} implies that, as soon as there is $1$, there is an infinity of Koopman eigenfunctions (however, they could be dependent). Also, it follows from Property \ref{property_eig} that the products $\phi_{\lambda_i}^{k_i}\phi_{\lambda_j}^{k_j}$, with $k_i=\lambda_j$ and $k_j=-\lambda_i$, satisfy \eqref{eig_evol} with $\lambda=0$. (They are eigenfunctions only if they belong to $\mathcal{F}$.) These functions are constant along the trajectories, so that their level sets are invariant under $\varphi$. If one considers the non-degenerate intersections of the level sets of $N-1$ such (independent) functions, we obtain a family of one-dimensional sets that correspond to the orbits of the system. This property shows that the Koopman eigenfunctions are directly related to the dynamics of the systems. More precisely, knowing them is equivalent to knowing the trajectories of the system.

The following property shows that the set of eigenfunctions can be decomposed in two subsets. Only one of these subsets is related to the stability properties of the system.

\begin{property}
\label{property_eig2}
Suppose that $A$ is globally stable on $X$. If $\mathcal{F} \subseteq C^0(X)$, the Koopman eigenfunctions $\phi_\lambda$ and their associated eigenvalues $\lambda$ satisfy
\begin{eqnarray}
\phi_\lambda \in \mathcal{F}_{A_c} & \Leftrightarrow & \Re\{ \lambda\}<0 \,, \label{real_eig_neg} \\
\phi_\lambda \notin \mathcal{F}_{A_c} & \Leftrightarrow & \Re\{ \lambda\}=0 \,. \label{real_eig_zero}
\end{eqnarray}
\end{property}
Moreover, if $\phi_\lambda \in \mathcal{F}_{A_c}$, then it is also an eigenfunction of the restriction $U^t_{A_c}$, associated with the same eigenvalue. If $\phi_{\lambda} \notin \mathcal{F}_{A_c}$, then the restriction $\phi_\lambda|_A$ of $\phi_{\lambda}$ to $A$ is an eigenfunction of $U_{A}^t$, associated with the same eigenvalue.
\begin{proof}
1. $\phi_\lambda \in \mathcal{F}_{A_c} \Rightarrow \Re\{ \lambda\}<0$. The property directly follows from Proposition \ref{prop_Koopman_stab}.\\
2. $\Re\{ \lambda\}<0 \Rightarrow \phi_\lambda \in \mathcal{F}_{A_c}$. Consider a point $x_\omega \in A$. There exists a state $x\in X$ such that $x_\omega \in \omega(x)$, or equivalently there exists a sequence $t_k$ such that $t_k \rightarrow \infty$ and $\varphi^{t_k}(x)\rightarrow x_\omega$ as $k \rightarrow \infty$. Then, the continuity of $\phi_\lambda$ and \eqref{eig_evol} imply that
\begin{equation}
\label{limit_phi}
\phi_\lambda(x_\omega)= \phi_\lambda\left(\lim_{n\rightarrow \infty} \varphi^{t_k}(x)\right)= \lim_{k\rightarrow \infty} \phi_\lambda(\varphi^{t_k}(x))=0\,.
\end{equation}
3. $\phi_\lambda \notin \mathcal{F}_{A_c} \Leftrightarrow \Re\{ \lambda\}=0$. Suppose that there exists an eigenfunction which satisfies $\Re\{\lambda\}>0$. According to \eqref{eig_evol}, we have
\begin{equation*}
\phi_\lambda \left(\lim_{t\rightarrow \infty} \varphi^{t}(x)\right)= \lim_{t\rightarrow \infty} e^{\lambda t} \phi_\lambda(x) =\infty
\end{equation*}
for some $x$ such that $\phi_\lambda(x) \neq 0$. The eigenfunction $\phi_\lambda$ is not bounded on $A$, which contradicts the continuity assumption. It follows that the eigenvalues always satisfy $\Re\{\lambda\}\leq 0$, so that \eqref{real_eig_zero} is equivalent to \eqref{real_eig_neg}.\\
4. The fact that $\phi_\lambda \in \mathcal{F}_{A_c}$ is an eigenfunction of $U^t_{A_c}$ is trivial since $\mathcal{F}_{A^c}$ is invariant under $U^t$. The fact that $\phi_\lambda|_A$ is an eigenfunction of $U^t_A$ follows from $U^t_A(\phi_\lambda|_A)=(U^t \phi_\lambda)|_A=e^{\lambda t} \phi_\lambda|_A$.
\end{proof}

The eigenfunctions $\phi_\lambda$ that do not belong to $\mathcal{F}_{A_c}$ are associated with purely imaginary eigenvalues and provide no information on stability. Instead, they are related to the dynamics on the attractor $A$. More precisely, their restrictions $\phi_\lambda|_A \in L^2(A)$ are the eigenfunctions of $U_A$, which is a unitary operator describing the ergodic behavior of the trajectories on $A$. In addition, the level sets of $\phi_\lambda$ are the sets of initial conditions \blue{converging to the same trajectory} on the attractor. They are closely related to the notion of periodic invariant sets \cite{Mezic} and to the \emph{isochrons} defining phase coordinates on the state space \cite{Mauroy_Mezic}.

In contrast, the eigenfunctions $\phi_\lambda$ that belong to $\mathcal{F}_{A_c}$ are associated with eigenvalues $\Re\{\lambda\}\neq 0$ and capture the stability properties of the system. They are also the eigenfunctions of the restriction $U_{A_c}$, a property which is in agreement with the results of Section \ref{sec_first_results} showing that $U_{A_c}$ plays a key role for stability analysis. The level sets of $|\phi_\lambda|$ are related to the notion of \emph{isostables}, i.e. the sets of initial conditions that converge synchronously toward the attractor \cite{MMM_isostables}.

\begin{remark}[Space of observables]
It has been shown recently in \cite{Mohr} that an appropriate space of observables---for which the Koopman operator \blue{admits a spectral expansion} (in the case of stable fixed points and limit cycles)---is \blue{(the completion of) a space of polynomials with indeterminates being coordinates corresponding to stable directions of the attractor and with coefficients being observables defined on the attractor.} For the sake of simplicity (and for practical reasons), we consider in this paper more general spaces (e.g. $C^0(X)$, $C^1(X)$) in which the operator might not \blue{admit a spectral expansion}, but which still ensure that the eigenfunctions capture the required stability properties of the system. In addition, the results of \cite{Mohr} motivate the choice of a polynomial basis for the numerical simulations proposed in Section \ref{sec_num_methods}.
\end{remark}

\begin{remark}[Continuous and regular spectrum]
For our purpose, we only need to consider the point spectrum of the Koopman operator. In well-chosen spaces of observables, the continuous and residual parts of the spectrum are empty with most of the types of \blue{hyperbolic} attractors (fixed point, limit cycle, quasiperiodic tori, see e.g. \cite{Gaspard,Gaspard2,MMM_isostables,Mohr}). For chaotic systems, these parts correspond to the ergodic dynamics on the strange attractor, therefore carrying no information on stability. \blue{When the attractor is not hyperbolic, a non empty continuous spectrum is also observed with the space of analytic functions (see e.g. \cite{Gaspard}). In this case, the associated (non-analytic) generalized eigenfunctions can be used for stability analysis (see also Remark \ref{rem_non_hyperb}).}
\end{remark}

\subsection{Main results}
\label{subsec_main_results}

As suggested by Property \ref{property_eig2}, the eigenfunctions lying in $\mathcal{F}_{A_c}$ (i.e., associated with $\Re\{\lambda\}<0$) can be used for the global stability analysis of the attractor. We have the following general result.
\begin{theorem}
\label{theo_stab_eig}
Suppose that $X$ is a forward invariant compact set and that the Koopman operator $U^t f = f \circ \varphi^t$ admits an eigenfunction $\phi_{\lambda} \in  C^0(X)$ with the eigenvalue $\Re\{\lambda\}<0$. Then the zero level set
\begin{equation*}
M_0=\{x \in X|\phi_{\lambda}(x)=0\}
\end{equation*}
is forward invariant under $\varphi^t$ and globally asymptotically stable.
\end{theorem}
\begin{proof}
\blue{
\emph{Invariance.} We define the set
\begin{equation*}
\Lambda_\alpha=\{x \in X \, | \, |\phi_\lambda(x)|<\alpha\} 
\end{equation*}
for some $\alpha>0$. For $x \in \Lambda_\alpha$, we have $|\phi_\lambda(\varphi^t(x))|=e^{\Re\{\lambda\} t} |\phi_\lambda(x)| < \alpha$ for all $t>0$, so that $\phi^t(x) \in \Lambda_\alpha$ for all $t \geq 0$. In particular, $M_0 = \Lambda_{\alpha=0}$ is forward invariant. \\
\emph{Stability.} For some $\epsilon>0$, consider the set
\begin{equation*}
\mathcal{N}_\epsilon = \{x \in X | \min_{y \in M_0} \|x-y\| < \epsilon \}\,.
\end{equation*}
By continuity of $\phi_\lambda$ and since $\phi_\lambda(x)=0$ for all $x \in M_0$, there exist $\alpha>0$ and $\delta>0$ such that $\mathcal{N}_\delta \subset \Lambda_\alpha \subset \mathcal{N}_\epsilon$. Since $\Lambda_\alpha$ is forward invariant, we have
\begin{equation*}
x \in \mathcal{N}_\delta \Rightarrow x \in \Lambda_\alpha \Rightarrow \varphi^t(x) \in \Lambda_\alpha \Rightarrow \varphi^t(x) \in \mathcal{N}_\epsilon \quad \forall t \geq 0
\end{equation*}
and $M_0$ is stable.\\
}
\emph{Global attractivity.} Since $\Re\{\lambda\}<0$, the equality \eqref{eig_evol} implies that $\lim_{t \rightarrow \infty} \phi_{\lambda}(\varphi^t(x))= 0$ $\forall x$. 
Since $\phi_{\lambda}$ is continuous, \eqref{limit_phi} holds so that the limit set of every trajectory is contained in $M_0$. This concludes the proof.
\end{proof}
The proof of the stability property is inspired from the proof of the Krasovskii-LaSalle principle (see e.g. \cite{Khalil}). The main difference is that the differentiability of the eigenfunctions is not required here since we know \emph{a priori} that these eigenfunctions are decreasing (and asymptotically converge to zero) along the trajectories. Also, the result holds with $X=\mathbb{R}^N$ provided that $\lim_{\| x \| \rightarrow \infty} |\phi_\lambda(x)| \neq 0$.

Theorem \ref{theo_stab_eig} implies the following corollary.
\begin{corollary}
\label{corol_stab_eig}
Suppose that $X$ is a forward invariant compact set and that the Koopman operator $U^t f = f \circ \varphi^t$ admits the eigenfunctions $\phi_{\lambda_i} \in  C^0(X)$ with the eigenvalues $\Re\{\lambda_i\}<0$, $i=1,\dots,m$. Then the intersection of the zero level sets
\begin{equation*}
M=\bigcap_{i=1}^m \{ x \in X | \phi_{\lambda_i}(x)=0 \}
\end{equation*}
is invariant under $\varphi^t$ and globally asymptotically stable.
\end{corollary}
\blue{
\begin{proof}
\emph{Invariance and global attractivity.} Let $M_0^i$ denote the zero level set of $\phi_{\lambda_i}$. Theorem \ref{theo_stab_eig} implies that $M_0^i$ is forward invariant and contains the limit sets $\omega(x)$ for all $x\in X$. Then, $M=\cap_i M_0^i$ is also forward invariant and contains the limit sets $\omega(x)$ for all $x\in X$.\\
\emph{Stability.} For some $\epsilon>0$, it is clear that there exists $\epsilon' \leq \epsilon$ such that
\begin{equation}
\label{relat_stability}
\min_{y \in M_0^i} \|z-y\|<\epsilon'\,\, \forall i\in\{1,\dots,m\} \Rightarrow \min_{y \in \cap_i M_0^i} \|z-y\|<\epsilon\,.
\end{equation}
Moreover, Theorem \ref{theo_stab_eig} implies that, for some $\epsilon'>0$, there exists $\delta>0$ such that
\begin{equation}
\label{relat_stability2}
\min_{y \in \cap_i M_0^i} \|x-y\|<\delta \Rightarrow \min_{y \in M_0^i} \|x-y\|<\delta \Rightarrow \min_{y \in M_0^i} \|\varphi^t(x)-y\|<\epsilon' \,\, \forall t \geq 0 
\end{equation}
for $i=1,\dots,m$. The proof follows from \eqref{relat_stability} with $z=\varphi^t(x)$ and \eqref{relat_stability2}.
\end{proof}
}
According to Proposition \ref{property_eig2}, the eigenfunctions considered in Corollary \ref{corol_stab_eig} are zero on the attractor and it follows that $A \subseteq M$. In order to prove the stability of $A$, several eigenfunctions are required so that the intersections of their zero level sets satisfy $M=A$ (typically $N-q$ eigenfunctions if the attractor is of dimension $q$). It is noticeable that, despite the fact that the Koopman operator is infinite-dimensional, only a (small) finite number of eigenfunctions is sufficient to establish global stability of the attractor.

\begin{remark}[Stable and unstable manifolds]
If $A$ is an invariant manifold but not an attractor, it is obviously not possible to obtain $M=A$, \blue{where $M$ is the intersection of the zero level sets of all the eigenfunctions that correspond to an eigenvalue with a strictly negative real part}. In this case, the smallest set $M$ is the unstable manifold of $A$, and the intersection $M_s$ of the zero level sets of the eigenfunctions $\phi_{\lambda_i}$, with $\Re\{\lambda_i\}>0$, is the stable manifold of $A$. For the Koopman operator acting on the space of functions $f|_{M_s}$ restricted to $M_s$, one can find an intersection of zero level sets of eigenfunctions $\phi_{\lambda_i}|_{M_s}$, with $\Re\{\lambda_i\}<0$, that is equal to $A$. In this case, Corollary \ref{corol_stab_eig} implies that $A$ is globally stable in $M_s$ (i.e. under the flow $\varphi^t|_{M_s}$).
\end{remark}

\blue{
\begin{remark}[Compact forward invariant set]
The result of Corollary \ref{corol_stab_eig} can also be used when $X$ is not compact forward invariant (or when this cannot be proved). If the (compact) set
\begin{equation*}
X_0 = \bigcap_{i=1}^m \{ x \in X | |\phi_{\lambda_i}(x)| < \alpha_i \}
\end{equation*}
satisfies $X_0 \cap \partial X = \emptyset$ (with appropriate values $\alpha_i$) where $\partial X$ is the boundary of $X$, then it is a (compact) forward invariant set. In this case, Corollary \ref{corol_stab_eig} implies global stability in $X_0$. Note that compact forward invariant sets are also given by the level sets of Lyapunov functions derived from the Koopman eigenfunctions (see Section \ref{subsub_fxpt}).
\end{remark}
}

In the following, the general result of Theorem \ref{theo_stab_eig}---and in particular Corollary \ref{corol_stab_eig}---is applied to study the stability of particular attractors, such as \blue{hyperbolic} fixed points and limit cycles. For these cases, we show that the local stability property of the attractors is extended to a global stability property through the Koopman eigenfunctions.\\

\subsubsection{The case of a \blue{hyperbolic} fixed point}
\label{subsub_fxpt}

When the attractor is a fixed point, the point spectrum of the Koopman operator captures the eigenvalues of the Jacobian matrix $J$ evaluated at the fixed point. The corresponding eigenfunctions are used with Corollary \ref{corol_stab_eig} and lead to a necessary and sufficient criterion for global stability.
\begin{proposition}
\label{prop_stab_fxpt}
Let $X\subset \mathbb{R}^N$ be a connected, forward invariant, compact set. Consider that \eqref{syst} with $F\in C^2(X)$ admits a fixed point $x^*\in X$ and assume that the Jacobian matrix $J$ of $F$ at $x^*$ has $N$ \blue{distinct} eigenvalues with strictly negative real part. Then, the fixed point $x^*$ is globally stable in $X$ if and only if the Koopman operator associated with \eqref{syst} has $N$ eigenfunctions $\phi_{\lambda_i}\in C^1(X)$ \blue{(with distinct eigenvalues $\lambda_i$)}, with $\Re\{\lambda_i \}<0$ and $\nabla \phi_{\lambda_i}(x^*)\neq 0$. In addition, the eigenvalues $\lambda_i$ are the eigenvalues of $J$.
\end{proposition}
\begin{proof}
\emph{Sufficiency.} Consider the first order Taylor approximations
\begin{equation*}
\phi_{\lambda_i}(x) = \phi_{\lambda_i}(x^*) + \nabla \phi_{\lambda_i}(x^*) (x-x^*) + o(\|x-x^*\|)
\end{equation*}
and
\begin{equation*}
F(x) = J (x-x^*) + o(\|x-x^*\|) \,.
\end{equation*}
(Note that $\phi_{\lambda_i}(x^*)=0$, according to Property \ref{property_eig2}.) Injecting these two approximations into \eqref{eig_equa} \blue{and considering $x-x^*=\|x-x^*\| e_X$, where $e_X$ is the unit vector in the direction of $x-x^*$, we obtain to first order (i.e. by taking $\|x-x^*| \rightarrow 0$)
\begin{equation*}
(J e_X) \cdot \nabla \phi_{\lambda_i}(x^*) = \lambda_i \nabla \phi_{\lambda_i}(x^*) e_X\,.
\end{equation*}
This equation is valid for every $e_X$, i.e. every unit vector, so that we get}
\begin{equation}
\label{gradient_eig}
J^T \nabla \phi_{\lambda_i}(x^*) = \lambda_i \nabla \phi_{\lambda_i}(x^*) \,.
\end{equation}
Since $\nabla \phi_{\lambda_i}(x^*)\neq 0$, the Koopman eigenvalue is an eigenvalue of the Jacobian matrix $J$. Moreover, $\nabla \phi_{\lambda_i}(x^*)$ is equal (up to a multiplicative constant) to the left eigenvector $w_i$ of $J$, or equivalently $\Re\{\nabla \phi_{\lambda_i}(x^*)\}$ and $\Im\{\nabla \phi_{\lambda_i}(x^*)\}$ (when $\lambda_i$ is complex) are respectively parallel to $\Re\{w_i\}$ and $\Im\{w_i\}$. \blue{Thus the zero level sets of $\Re\{\phi_{\lambda_i}\}$ and $\Im\{\phi_{\lambda_i}\}$ are tangent at $x^*$ to a hyperplane whose normal is $\Re\{w_i\}$ and $\Im\{w_i\}$, respectively. Since the $N$ eigenvectors $\Re\{w_i\}$ and $\Im\{w_i\}$ (here we use a shorthand, the imaginary
part could be zero) are independent, the intersection in a small neighborhood $V$ of $x^*$ of the zero level sets of the $N$ different functions $\Re\{\phi_{\lambda_i}\}$ and $\Im\{\phi_{\lambda_i}\}$ can only be $x^*$ (to prove by contradiction, just take a sequence of points $\{x_i\}$ in a nested sequence of neighborhoods $V_i$ such
that $\Re\{\phi_{\lambda_i}(x_i)\}=0$ and $\Im\{\phi_{\lambda_i}(x_i)\}=0$, for every $i$. By taking the
limit of along that sequence, it is easy to see that gradients must be dependent.)}

Next, we show that the zero level sets cannot have another intersection in $X$. Suppose that there is another intersection. This must be a part of another invariant set which is not connected to the fixed point. Therefore, the boundary $\partial \Omega$ of the basin of attraction $\Omega$ of $x^*$ has a non empty intersection with $X$, since $X$ is a connected set. Moreover, since $X$ is forward invariant, $\partial \Omega \cap X$ is also forward invariant and contains the limit sets of its trajectories. Consider a point $x_\omega \in \partial \Omega \cap X$ that belongs to a limit set. By definition and continuity of the eigenfunctions, we have $\phi_{\lambda_i}(x_\omega)=0$ $\forall i$ (see \eqref{limit_phi}). Also, for any arbitrarily small neighborhood $V_\epsilon$ of $x_\omega$ and for all $x_\epsilon \in V_\epsilon \cap \Omega$, there exist a point $x_V \in V$ with $x_V\neq x^*$ and a constant $T>0$ such that $\varphi^{-T}(x_V)=x_\epsilon$. There is at least one eigenfunction that satisfies $|\phi_{\lambda_i}(x_V)|=C>0$,
since $x^*$ is the only intersection in $V$ of the zero level sets of the $N$ eigenfunctions. Equivalently we have $|\phi_{\lambda_i}(x_\epsilon)|=C \exp(-\Re\{\lambda_i\} T)>0$. Therefore, $\phi_{\lambda_i}$ is not continuous in $V_\epsilon \subset X$, which is a contradiction.

Finally, since $\Re\{\lambda_i\}<0$, the result follows from Corollary \ref{corol_stab_eig} with $M=\{x^*\}$.

\emph{Necessity.} (The proof is inspired from results presented in \cite{Lan}.) Since the fixed point is globally stable in $X$, it follows from Theorem 2.3 in \cite{Lan} that there exists a $C^1$ diffeomorphism $y=h(x)$ such that $\dot{y}=J \, y$, $h(x^*)=0$, and the Jacobian matrix of $h$ at $0$ satisfies $J_h=I$. For this linear system, there exist $N$ distinct Koopman eigenfunctions $\tilde{\phi}_{\lambda_i}(y)=y \cdot w_i$ that are associated with the eigenvalues of $J$ (see e.g. \cite{MMM_isostables,Mezic_ann_rev}). It follows that the Koopman operator of \eqref{syst} has $N$ $C^1$ eigenfunctions of the form $\phi_{\lambda_i}=\tilde{\phi}_{\lambda_i} \circ h$ (with the same eigenvalues). Moreover, we have $\nabla \phi_{\lambda_i}(x^*)=J_h^T \, w_i=w_i \neq 0$. This concludes the proof.
\end{proof}
Proposition \ref{prop_stab_fxpt} can be interpreted as the global equivalent of the well-known local stability result for a fixed point. While local stability depends on $N$ eigenvalues of the Jacobian matrix, global stability depends on $N$ $C^1$ Koopman eigenfunctions (associated with the same eigenvalues). The support of these particular eigenfunctions corresponds to the basin of attraction of the fixed point. Note also that the condition $\nabla \phi_\lambda(x^*) \neq 0$ is necessary to rule out the Koopman eigenfunctions of the form $\phi_{\lambda_1}^{k_1}\cdots \phi_{\lambda_N}^{k_N}$ (see Property \ref{property_eig}), which are redundant with respect to the basic eigenfunctions $\phi_{\lambda_i}$.

\begin{remark}[Linear systems] In the case of a linear system $\dot{x}=A x$, the \blue{eigenfunctions} are given by $\phi_{\lambda_i}(x)=x \cdot w_i$ where $w_i$ is a left eigenvector of $A$. They satisfy $\phi_{\lambda_i} \in C^1(\mathbb{R}^N)$ and $\nabla \phi_{\lambda_i}(x^*) = w_i \neq 0$, so that Proposition \ref{prop_stab_fxpt} only requires $\Re\{\lambda_i\}<0$. We recover the usual stability criterion for (global) stability of linear systems.
\end{remark}

\blue{
\begin{remark}[Non hyperbolic case]
\label{rem_non_hyperb}
The result of Proposition \ref{prop_stab_fxpt} cannot be used when the fixed point is not hyperbolic (i.e. the Jacobian matrix $J$ has at least one eigenvalue $\Re\{\lambda_i\}=0$). In this case, the Koopman operator does not possess $N$ eigenfunctions that satisfy the assumptions of Proposition \ref{prop_stab_fxpt} (i.e. $\Re\{\lambda \}<0$ and $\nabla \phi_{\lambda}(x^*)\neq 0$). However, the operator has a continuous spectrum with eigenvalues satisfying $\Re\{\lambda\}<0$ (see \cite{Gaspard}), so that Theorem \ref{theo_stab_eig} can be used. For instance, the system $\dot{x}=-x^3$ admits a continuous (generalized) eigenfunction $\phi_\lambda(x)=\exp(-1/(2x^2))$ with the associated eigenvalue $\lambda=-1$, and Theorem \ref{theo_stab_eig} implies global stability of the (non hyperbolic) fixed point at the origin. Note that $d^n\phi_\lambda/dx^n(0)=0$ for all $n\in \mathbb{N}$ so that the eigenfunction is not analytic.
\end{remark}
}

\blue{For the sufficiency part of Proposition \ref{prop_stab_fxpt}, the assumptions on the eigenvalues of the Jacobian matrix are actually not required. We have the following Corollary.
\begin{corollary}
\label{corol_stab_fxpt}
Let $X\subset \mathbb{R}^N$ be a connected, forward invariant, compact set and consider that \eqref{syst} admits a fixed point $x^*\in X$. If there exist $N$ Koopman eigenfunctions $\phi_{\lambda} \in C^1(X)$, with $\Re\{\lambda\}<0$ and such that the $N$ vectors $\nabla \phi_\lambda(x^*) \neq 0$ are linearly independent, then $x^*$ is globally asymptotically stable in $X$.
\end{corollary}
\begin{proof}
The result follows directly from Corollary \ref{corol_stab_eig} and from the proof of Proposition \ref{prop_stab_fxpt}.
\end{proof}
}

The Koopman eigenfunctions considered in Proposition \ref{prop_stab_fxpt} \blue{and Corollary \ref{corol_stab_fxpt}} yield the Lyapunov functions \cite{MauroyMezic_CDC,MMM_isostables}
\begin{equation}
\label{Lyap_fct}
\mathcal{V}(x)=\left( \sum_{i=1}^N |\phi_{\lambda_i}(x)|^p\right)^{1/p}
\end{equation}
with the integer $p\geq 1$. According to \eqref{eig_evol}, these Lyapunov functions satisfy $\mathcal{V}(\varphi^t(x)) \leq \exp(\Re\{\lambda_1\}t) \mathcal{V}(x)$ for $x\in X$, where $\lambda_1$ is the eigenvalue closest to the imaginary axis. (Note that the Lyapunov functions are not necessarily smooth.) \blue{This result can be used to define a (compact) set
\begin{equation*}
\Omega_\alpha = \{ x \in X | V(x) < \alpha \}
\end{equation*}
which is forward invariant if $\Omega_\alpha \cap \partial X = \emptyset$, where $\partial X$ is the boundary of $X$. This is useful when it cannot be shown that $X$ is forward invariant. In numerical simulations, the Lyapunov function \eqref{Lyap_fct} can also be used} to estimate the basin of attraction or to check the accuracy of the computations.

In addition, it is shown in \cite{MauroyMezic_CDC,MMM_isostables} that the eigenfunctions are related to (non-quadratic) metrics that are exponentially contracting on $X$. These metrics could be considered through the differential framework recently developed in \cite{Forni}.

\subsubsection{The case of a \blue{hyperbolic} limit cycle}

When the attractor is a limit cycle, the point spectrum of the Koopman operator captures the Floquet exponents. The corresponding eigenfunctions are used with Corollary \ref{corol_stab_eig} and lead to a necessary and sufficient criterion for global stability. Note that the point spectrum also contains imaginary eigenvalues of the form $\lambda=i k 2\pi/T$, with $k\in \mathbb{Z}$ and where $T$ is the period of the limit cycle, but the corresponding eigenfunctions are not related to the stability of the system.

\begin{proposition}
\label{prop_stab_limcyc}
Let $X\subset \mathbb{R}^N$ be a connected, forward invariant, compact set. Consider that \eqref{syst} with $F\in C^2(X)$ admits a limit cycle $\Gamma \subset X$ and assume that the monodromy matrix evaluated at some $x^\gamma\in \Gamma$ has $N-1$ \blue{distinct} eigenvalues (Floquet exponents) with strictly negative real part associated with (Floquet) eigenvectors $v_i$. Then, the limit cycle $\Gamma$ is globally stable in $X$ if and only if the Koopman operator associated with \eqref{syst} has $N-1$ eigenfunctions $\phi_{\lambda_i}\in C^1(X)$ \blue{(with distinct eigenvalues $\lambda_i$)} with $\Re\{\lambda_i\}<0$, and such that $\nabla \phi_{\lambda_i}$ is differentiable along the limit cycle and satisfies $\nabla \phi_{\lambda_i}(x^\gamma) \cdot v_i \neq 0$. In addition, the eigenvalues $\lambda_i$ are the Floquet exponents of the limit cycle.
\end{proposition}
\begin{proof}
\emph{Sufficiency.} Since the limit cycle is hyperbolic, there exists a $C^2$ local change of coordinates $y=y(x) \in \mathbb{R}^{n-1}$, $\theta=\theta(x) \in \mathbb{S}^1$, such that the dynamics become
\begin{equation}
\label{new_dyn_lim_cycle}
\dot{y}=G(y,\theta) \,, \qquad \dot{\theta}=\omega
\end{equation}
in the neighborhood $V$ of the limit cycle. The level sets of the $\theta$ coordinate are the isochrons and the $y$ coordinates are related to the directions transverse to the limit cycle (i.e., tangent to the isochrons). Moreover, we have $\omega=2\pi/T$ where $T$ is the period of the limit cycle, $y(x)=0$ for all $x \in \Gamma$, and without loss of generality $\theta(x^\gamma)=0$. The Koopman eigenfunctions $\tilde{\phi}_{\lambda_i}$ related to \eqref{new_dyn_lim_cycle} satisfy $\tilde{\phi}_{\lambda_i}(y(x),\theta(x))=\phi_{\lambda_i}(x)$. Next, we consider the first order Taylor approximations
\begin{equation*}
\tilde{\phi}_{\lambda_i}(y,\theta) = \tilde{\phi}_{\lambda_i}(0,\theta) + \nabla \tilde{\phi}_{\lambda_i}(0,\theta) \, y + o(\|y\|)
\end{equation*}
and
\begin{equation*}
G(y,\theta) = J_G(\theta) \, y + o(\|y\|) \,,
\end{equation*}
where $\nabla \tilde{\phi}_{\lambda_i}=(\partial \tilde{\phi}_{\lambda_i}/\partial y_1,\dots,\partial \tilde{\phi}_{\lambda_i}/\partial y_{n_1})$ and $J_G(\theta)$ is the Jacobian matrix of $G$ for $y=0$. (Note that $\tilde{\phi}_{\lambda_i}(0,\theta)=0$, according to Property \ref{property_eig2}.) Injecting these two approximations into \eqref{eig_equa}, we obtain for the first order terms
\begin{equation*}
J_G^T(\theta) \nabla \tilde{\phi}_{\lambda_i}(0,\theta) +\omega \frac{d \nabla \tilde{\phi}_{\lambda_i}}{d\theta} (0,\theta)  = \lambda_i \nabla \tilde{\phi}_{\lambda_i}(0,\theta)\,.
\end{equation*}
where the derivative of $\nabla \tilde{\phi}_{\lambda_i}(0,\theta)$ is well-defined according to the assumption. The solution is given by
\begin{equation}
\label{sol_Floquet}
\nabla \tilde{\phi}_{\lambda_i}(0,\theta) = \Psi(\theta) \nabla \tilde{\phi}_{\lambda_i}(0,0) \, e^{\lambda_i \theta / \omega}
\end{equation}
where the fundamental matrix $\Psi(\theta)$ satisfies $\Psi(0)=I$ and
\begin{equation*}
\frac{d \Psi}{d\theta}=-\frac{1}{\omega}J_G^T \, \Psi\,.
\end{equation*}
We remark that the differentiation of $\Psi^{-1} \Psi = I$ yields
\begin{equation*}
\frac{d \Psi^{-1}}{d\theta}=-\Psi^{-1} \frac{d\Psi}{d\theta} \Psi^{-1}=\frac{1}{\omega} \Psi^{-1} J_G^T
\end{equation*}
or equivalently
\begin{equation}
\label{def_monodromy}
\frac{d \Phi}{d\theta} = \frac{1}{\omega} J_G\, \Phi
\end{equation}
with $\Phi=\Psi^{-T}$. We can rewrite \eqref{sol_Floquet} as
\begin{equation*}
\Phi^T(\theta) \nabla \tilde{\phi}_{\lambda_i}(0,\theta) = \nabla \tilde{\phi}_{\lambda_i}(0,0) \, e^{\lambda_i \theta / \omega}
\end{equation*}
and for $\theta=2\pi$, we obtain
\begin{equation}
\label{Floquet_2pi}
\Phi^T(2\pi) \nabla \tilde{\phi}_{\lambda_i}(0,0) = \nabla \tilde{\phi}_{\lambda_i}(0,0) \, e^{\lambda_i T}\,.
\end{equation}
Since the gradient satisfies $\nabla \phi_{\lambda_i}(x^\gamma) \cdot v_i \neq 0$, it has a component tangent to the isochron, in the transverse direction related to $y$, so that $\nabla \tilde{\phi}_{\lambda_i}(0,0) \neq 0$. The relationship \eqref{Floquet_2pi} then implies that $\exp(\lambda_i T)$ is an eigenvalue of the monodromy matrix $\Phi(2\pi)$. It follows from \eqref{def_monodromy} that $\exp(\lambda_i T)$ is a Floquet multiplier of the limit cycle, or equivalently $\lambda_i$ is a Floquet exponent. In addition, $\nabla \tilde{\phi}_{\lambda_i}(0,0)$ is a left eigenvector of $\Phi(2\pi)$, so that it is perpendicular to $N-2$ Floquet vectors $v_j$, with $j\neq i$. Since the Floquet vectors are independent, the intersection in $V$ of the zero level sets of the $N-1$ different eigenfunctions $\phi_{\lambda_i}$ can only be $\Gamma$, and there is no other intersection in $X$. (A detailed proof is not repeated here but follows similar lines as the proof of Proposition \ref{prop_stab_fxpt}.)

Finally, since $\Re\{\lambda_i\}<0$, the result follows from Corollary \ref{corol_stab_eig} with $M=\Gamma$.

\emph{Necessity.} (The proof is inspired from results presented in \cite{Lan}.) Since the limit cycle is globally stable in $X$, it follows from Theorem 2.6 in \cite{Lan} that \eqref{syst} is conjugated to
\begin{equation}
\label{lin_lim_cycle}
\dot{z}=B \, z \,, \qquad \dot{\theta}=\omega
\end{equation}
through a $C^1$ diffeomorphism $(z,\theta)=h(x)=(h_1(x),h_2(x))$ such that $h_1(x)=0$ for all $x \in \Gamma$. The eigenvalues of the $(N-1)\times(N-1)$ matrix $B$ are the $N-1$ stable Floquet exponents $\lambda_i$. For \eqref{lin_lim_cycle}, there exist $N-1$ Koopman eigenfunctions $\tilde{\phi}_{\lambda_i}(z,\theta)=z \cdot w_i$, where $w_i$ are the left eigenvectors of $B$, which are associated with the Floquet exponents $\lambda_i$. It follows that the Koopman operator of \eqref{syst} has $N-1$ $C^1$ eigenfunctions of the form $\phi_{\lambda_i}=\tilde{\phi}_{\lambda_i} \circ h$, or equivalently $\phi_{\lambda_i}=h_1(x) \cdot w_i$ (with the same eigenvalues). Moreover, using the chain rule, we have
\begin{equation*}
\nabla \phi_{\lambda_i}(x^\gamma)=\left(\nabla_z \tilde{\phi}_{\lambda_i} (h(x^\gamma))\right)^T \, J_{h_1}(x^\gamma) + \frac{\partial \tilde{\phi}_{\lambda_i}}{\partial \theta}(h(x^\gamma)) \nabla h_2(x^\gamma) = w_i^T \, J_{h_1}(x^\gamma)\,,
\end{equation*}
where $J_{h_1}$ is the Jacobian matrix of $h_1$ and with $\nabla_z \tilde{\phi}= (\partial \tilde{\phi}/\partial z_1,\dots, \partial \tilde{\phi}/\partial z_{N-1})$. It follows that $\nabla \phi_{\lambda_i}(x^\gamma)$ is a left eigenvector of the monodromy matrix $J^{-1}_{h_1}(x^\gamma) e^{2\pi B/\omega} J_{h_1}(x^\gamma)$ in the $x$ coordinates, so that it satisfies $\nabla \phi_{\lambda_i}(x^\gamma) \cdot v_i \neq 0$. In addition, it is easy to see that, similarly to \eqref{sol_Floquet}, we have
\begin{equation}
\label{sol_Floquet_bis}
\nabla \phi_{\lambda_i}(\varphi^t(x^\gamma)) = \Psi(t) \nabla \phi_{\lambda_i}(x^\gamma) \, e^{\lambda_i t}
\end{equation}
with the fundamental matrix $\Psi(t)$ satisfying $d \Psi/dt=-J^T(\varphi^t(x^\gamma)) \, \Psi$ and $\Psi(0)=I$. Since the Jacobian matrix $J$ of $F$ is $C^1$, $\Psi(t)$ is also $C^1$ and it follows from \eqref{sol_Floquet_bis} that $\nabla \phi_{\lambda_i}$ is differentiable along $\Gamma$. This concludes the proof.
\end{proof}
As in the case of a fixed point, the result is the global equivalent of the well-known local stability result. While local stability depends on the eigenvalues of the monodromy matrix (Floquet exponents), global stability depends on $C^1$ Koopman eigenfunctions (associated with the same eigenvalues). 
\blue{In addition, the sufficiency part can be made stronger, with no assumption on the Floquet eigenvalues and eigenvectors. We have the following Corollary.
\begin{corollary}
\label{corol_stab_limcyc}
Let $X\subset \mathbb{R}^N$ be a connected, forward invariant, compact set and consider that \eqref{syst} admits a limit cycle $\Gamma \subset X$. If there exist $N-1$ Koopman eigenfunctions $\phi_{\lambda} \in C^1(X)$, with $\Re\{\lambda\}<0$ and such that, for some $x^\gamma \in \Gamma$, the $N-1$ vectors $\nabla \phi_\lambda(x^\gamma) \neq 0$ are linearly independent, then $\Gamma$ is globally asymptotically stable in $X$.
\end{corollary}
\begin{proof}
The result follows directly from Corollary \ref{corol_stab_eig} and from the proof of Proposition \ref{prop_stab_limcyc}.
\end{proof}
}

We remark that Proposition \ref{prop_stab_limcyc} \blue{and Corollary \ref{corol_stab_limcyc}} cannot be applied if $X$ contains an unstable fixed point. The Koopman eigenfunctions are not $C^1$ at the unstable fixed point and the attractor is obviously not globally stable on $X$. Instead, Proposition \ref{prop_stab_limcyc} \blue{and Corollary \ref{corol_stab_limcyc}} must be considered with a set $X$ that does not contain a small disk centered at the unstable fixed point. The remark is also valid for Proposition \ref{prop_stab_fxpt} \blue{and Corollary \ref{corol_stab_fxpt} (see also Example 3 in \cite{MauroyMezic_CDC}).}

\section{Numerical methods}
\label{sec_num_methods}

The results of Section \ref{subsec_main_results} show the close relationship between particular Koopman eigenfunctions and global stability. In this section, we propose numerical techniques to compute these particular eigenfunctions, providing systematic ways to estimate the basin of attraction of the attractor or to establish global stability on a given subset of the state space. While existing methods for computing Koopman eigenfunctions rely on the evaluation of Laplace averages along the trajectories of the system (see e.g. \cite{MMM_isostables,Mezic_ann_rev}), the numerical schemes proposed here do not require the integration of trajectories. In accordance with the results of \cite{Mohr}, they rely on the expansion on a basis of polynomials. We use two different bases---Taylor and Bernstein polynomials---and consider separately the case of a stable fixed point and a stable limit cycle.

\subsection{Taylor expansion-based method for the fixed point}
\label{subsec_Taylor}

\blue{In this section,} we assume that the vector field $F$ is analytic and that the eigenvalues $\lambda_i$ of the Jacobian matrix at the fixed point $x^*$ are nonresonant \blue{(i.e. there do not exist integers $c_k \geq 0$ such that $\sum_{k} c_k \geq 2$ and $\lambda_i=\sum_k c_k \lambda_k$).} In this case, the Koopman eigenfunction admits a Taylor decomposition (at least in some neighborhood of $x^*$)
\begin{equation}
\label{Taylor_exp_phi}
\phi_{\lambda}(x)=\sum_{(k_1,\dots,k_N)\in \mathbb{N}^N}\phi_\lambda^{(k_1,\dots,k_N)} (x_1-x_1^*)^{k_1} \cdots (x_N-x_N^*)^{k_N}
\end{equation}
with
\begin{equation*}
\phi_\lambda^{(k_1,\dots,k_N)}=\frac{1}{k_1!\cdots k_N!} \left.\frac{\partial^{k_1+\cdots+k_N} \phi_\lambda}{\partial x_1^{k_1} \cdots \partial x_N^{k_N}} \right|_{x^*}
\end{equation*}
and with $x=(x_1,\dots,x_N)$. Similarly, the vector field $F(x)=(F_1(x),\dots,F_N(x))$ can be written as
\begin{equation}
\label{Taylor_exp_F}
F_l(x)=\sum_{(k_1,\dots,k_N)\in \mathbb{N}^N} F_l^{(k_1,\dots,k_N)} (x_1-x_1^*)^{k_1} \cdots (x_N-x_N^*)^{k_N}
\end{equation}
with
\begin{equation*}
F_l^{(k_1,\dots,k_N)}=\frac{1}{k_1!\cdots k_N!} \left.\frac{\partial^{k_1+\cdots+k_N} F_l}{\partial x_1^{k_1} \cdots \partial x_N^{k_N}} \right|_{x^*}\qquad l=1,\dots,n\,.
\end{equation*}
This decomposition in a basis of monomials can be regarded as the equivalent of the description of a measure through its moments in the dual Perron-Frobenius framework \cite{Lasserre_book}.

\blue{The monomials $(x_1-x_1^*)^{k_1} \cdots (x_N-x_N^*)^{k_N}$ can be represented as the components of the (infinite-dimensional) vector $X(x)$} and we rewrite \eqref{Taylor_exp_phi} as the product
\begin{equation}
\label{{Taylor_exp_phi}_vec}
\phi_\lambda(x) = \Phi^T \, X(x)
\end{equation}
where $\Phi$ is the vector containing the values $\phi_\lambda^{(k_1,\dots,k_N)}$.

Using basic properties of monomials, we can rewrite the eigenvalue equation \eqref{eig_equa} as
\begin{equation*}
\left(\sum_{l=1}^N \bar{M}^l \, \bar{D}^l \, \Phi \right)^T X(x) = \lambda \, \Phi^T \, X(x)
\end{equation*}
or equivalently
\begin{equation}
\label{num_Taylor}
\sum_{l=1}^N \bar{M}^l \,
\bar{D}^l \, \Phi = \lambda \, \Phi
\end{equation}
with
\begin{itemize}
\item the multiplication matrices
\begin{equation*}
\bar{M}^l=\sum_{k_1=0}^{\infty} \cdots \sum_{k_N=0}^{\infty} F_l^{(k_1,\dots,k_N)} M^{k_1} \otimes \cdots \otimes M^{k_N}
\end{equation*}
where $M^{k_l}$ are (infinite-dimensional) matrices with entries $M^{k_l}_{ij}=1$ if $j=i-k_l$ and $M^{k_l}_{ij}=0$ otherwise;
\item the differentiation matrices
\begin{equation*}
\bar{D}^l= \overbrace{I \otimes \cdots \otimes I}^{l-1 \textrm{ times}} \otimes D \otimes \overbrace{I \otimes \cdots \otimes I}^{N-l \textrm{ times}} \,,
\end{equation*}
where $I$ is the (infinite-dimensional) identity matrix and $D$ is a (infinite-dimensional) matrix with entries $D_{ij}=1$ if $i=j-1$ and $D_{ij}=0$ otherwise.
\end{itemize}
In order to solve the infinite-dimensional equation \eqref{num_Taylor}, we can consider a vector $\Phi^{(s)}$ whose components are the coefficients $\phi_\lambda^{(k_1,\dots,k_N)}$ related to the $s$th order in the Taylor expansion, i.e. $\sum_{i=1}^N k_i = s$. For each $s \in \mathbb{N}$, \eqref{num_Taylor} yields the $(s+N-1)!/((N-1)!s!)$-dimensional equation
\begin{equation}
\label{num_Taylor2}
H^{(s)} \Phi^{(s)} = \lambda \Phi^{(s)} + V^{(s)}
\end{equation}
where the coefficients $\phi_\lambda^{(k_1,\dots,k_N)}$ appearing in the expression of $V_i^{(s)}$ are related to an order smaller than $s$. Since the vector $V^{(s)}$ only depends on vectors $\Phi^{(s')}$ with $s'<s$, \eqref{num_Taylor2} can be solved recursively for increasing values $s$, a method which resembles a Carleman embedding method \cite{Carleman}. For $s=0$, we have $H^{(0)}=V^{(0)}=0$ since $F_l^{(0,\dots,0)}=0$ (i.e. $F_l(x^*)=0$). This yields the trivial solution $\phi_\lambda^{(0,\dots,0)}=0$ (i.e. $\phi_\lambda(x^*)=0$) . For $s=1$, we have $\Phi^{(1)}=\nabla \phi_\lambda(x^*)$, $H^{(1)}=J^T$ and $V^{(1)}=0$, so that \eqref{num_Taylor2} is equivalent to \eqref{gradient_eig}. In order to satisfy the assumption of Proposition \ref{prop_stab_fxpt}, the gradient $\nabla \phi_\lambda(x^*)$ must be nonzero, and is therefore a left eigenvector $w_i$ of $J$. The Koopman eigenvalue $\lambda$ is the corresponding eigenvalue $\lambda_i$ of the Jacobian matrix. This is in agreement with the results of Section \ref{subsub_fxpt}.

\paragraph*{Estimation of the basin of attraction}

According to the result of Proposition \ref{prop_stab_fxpt}, we can investigate the global stability of a fixed point by computing the Koopman eigenfunctions with the above numerical method. In addition, the method can be used to estimate the basin of attraction of the equilibrium. This is performed as follows: (i) build a candidate Lyapunov function of the form \eqref{Lyap_fct}, (ii) consider the region where this function is decreasing along the trajectories, and (iii) the largest closed level set of the Lyapunov function which is included in that region provides an inner approximation of the basin of attraction. This procedure can also be used to verify the results obtained with Taylor polynomials of small degree and to yield an approximate (conservative) region of stability.

\begin{example}
\label{ex3}
The dynamics
\begin{eqnarray*}
\dot{x}_1 & = & -x_2 \\
\dot{x}_2 & = & x_1-x_2+x_1^2 x_2
\end{eqnarray*}
of the well-known Van der Pol oscillator, here in backward-time, admit an unstable limit cycle which corresponds to the boundary of the basin of attraction of the stable origin. As shown in Figure \ref{eigenfct_ex34}(a), the method provides an accurate estimation of the basin of attraction.
\end{example}

\begin{example}
\label{ex4}
The dynamics
\begin{eqnarray}
\label{vdp_rev1}
\dot{x}_1 & = & x_2 \\
\label{vdp_rev2}
\dot{x}_2 & = & -2 x_1+\frac{1}{3} x_1^3-x_2
\end{eqnarray}
is characterized by a locally stable fixed point at the origin and two unstable saddle points at $(\pm\sqrt{6},0)$. Figure \ref{eigenfct_ex34}(b) shows that the numerical method provides a good estimation of the basin of attraction (delimited by the stable manifold of the saddle points). However, the largest approximation is obtained with a $14$th-order Taylor expansion and the method cannot capture the complete geometry of the basin. As explained below, this is due to the fact that the eigenfunctions admit a singularity at the saddle points and are analytic only on a subset of the basin of attraction.
\end{example}

\begin{figure}[h]
   \centering
   \subfigure[Example \ref{ex3}]{\includegraphics[height=5cm]{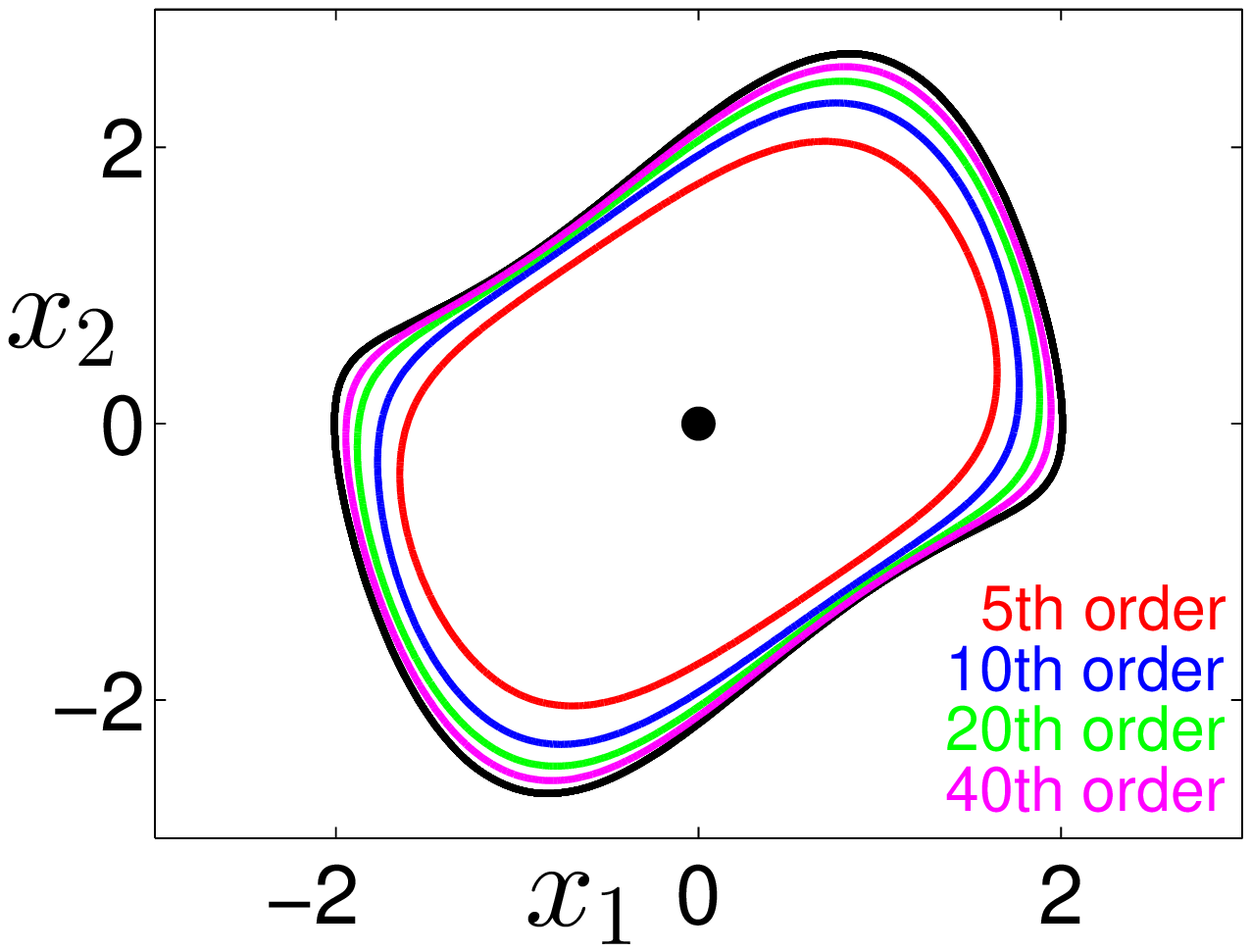}}
   \subfigure[Example \ref{ex4}]{\includegraphics[height=5cm]{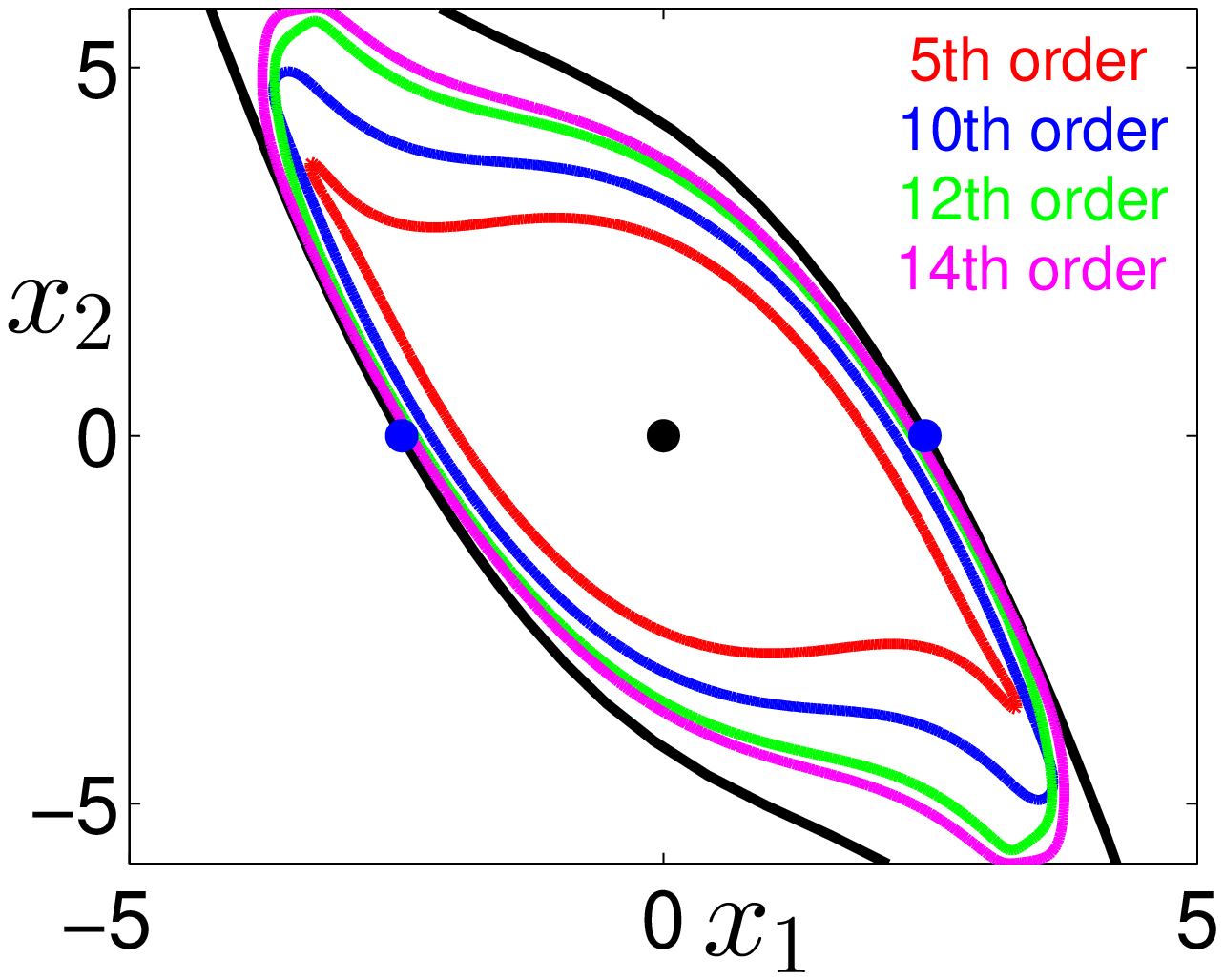}}
   \caption{The Taylor expansion-based method can be used to estimate the basin of attraction (black curve) of the stable fixed point (black dot). (a) Better approximations are obtained as the degree of the Taylor polynomials is increased. (b) The best result is obtained with a $14$th-order Taylor expansion. (In the two examples, $\lambda_1$ is complex and the Lyapunov function is $\mathcal{V}=|\phi_{\lambda_1}|=|\phi_{\lambda_2}|$.)}
   \label{eigenfct_ex34}
\end{figure}

\begin{example}[Non-analytic eigenfunctions]
\label{ex1}
For the dynamics
\begin{eqnarray*}
\dot{x}_1 & = & -3/4 x_1-1/8 x_2+1/4 x_1 x_2-1/4 x_2^2-1/2 x_1^3\,, \\
\dot{x}_2 & = & -1/8 x_1-x_2\,,
\end{eqnarray*}
the origin is globally stable in $X=[-2,2] \times [-2,2]$. \blue{However, the Taylor expansion of the eigenfunctions diverges in $X$, so that the method cannot prove global stability in this region. This is explained by the fact that the eigenfunctions are not analytic on $X$. Poincaré linearization theorem implies that the Koopman eigenfunctions are analytic on the \blue{largest} complex ball centered at the stable fixed point $x^*$ that does not contain a zero $z^* \neq x^* $ of $F(z)$ (see e.g. \cite{Gaspard}). Here, the vector field admits two complex zeros $z^* \approx -0.035 \pm 1.211i$ (acting as fictitious fixed points for the vector field in $\mathbb{C}$), so that the disk of analyticity of the eigenfunctions has a radius $\|z^*\|\approx 1.212$.}
\end{example}

\begin{figure}[h]
   \centering
   \subfigure{\includegraphics[width=7cm]{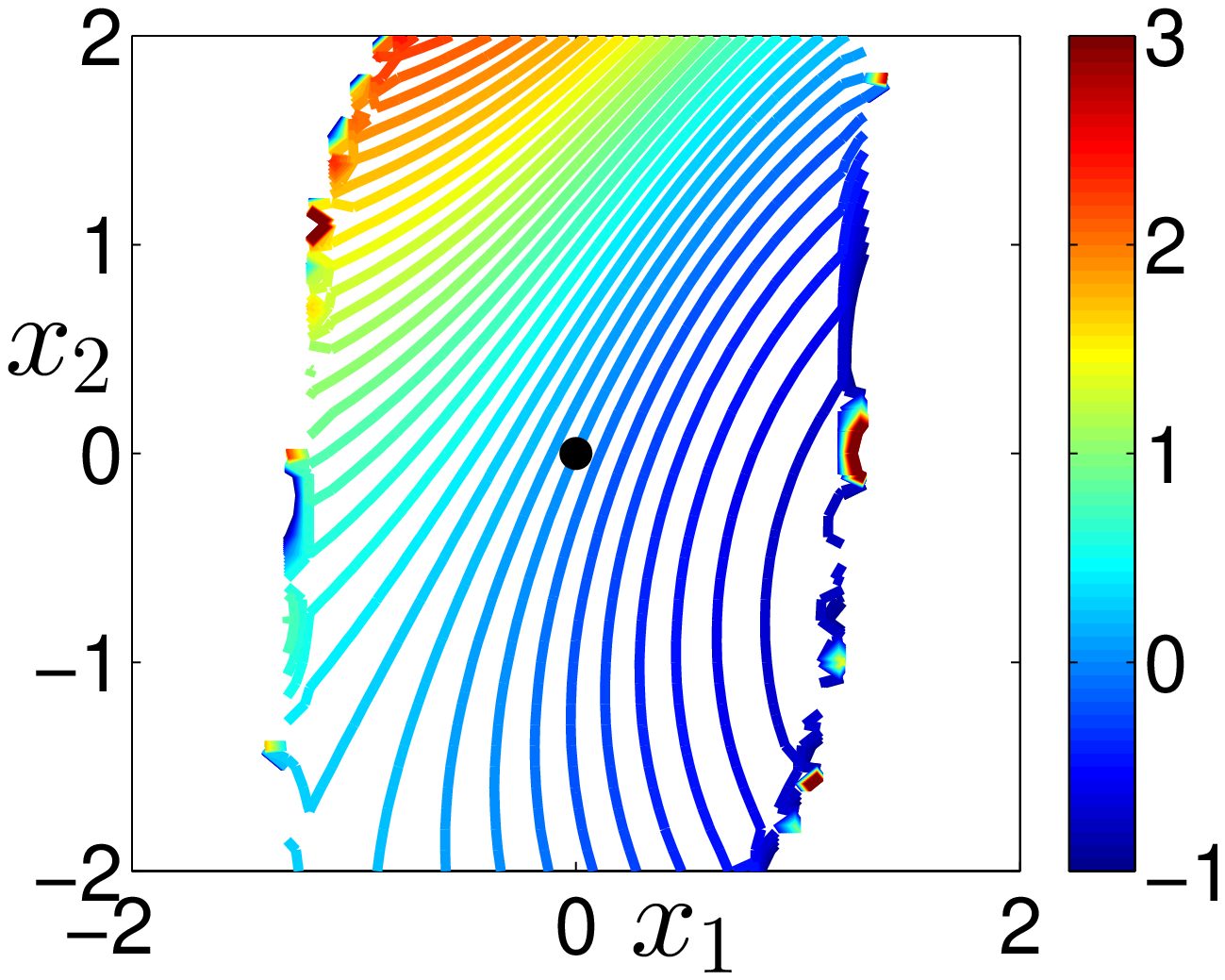}}
   \subfigure{\includegraphics[width=7cm]{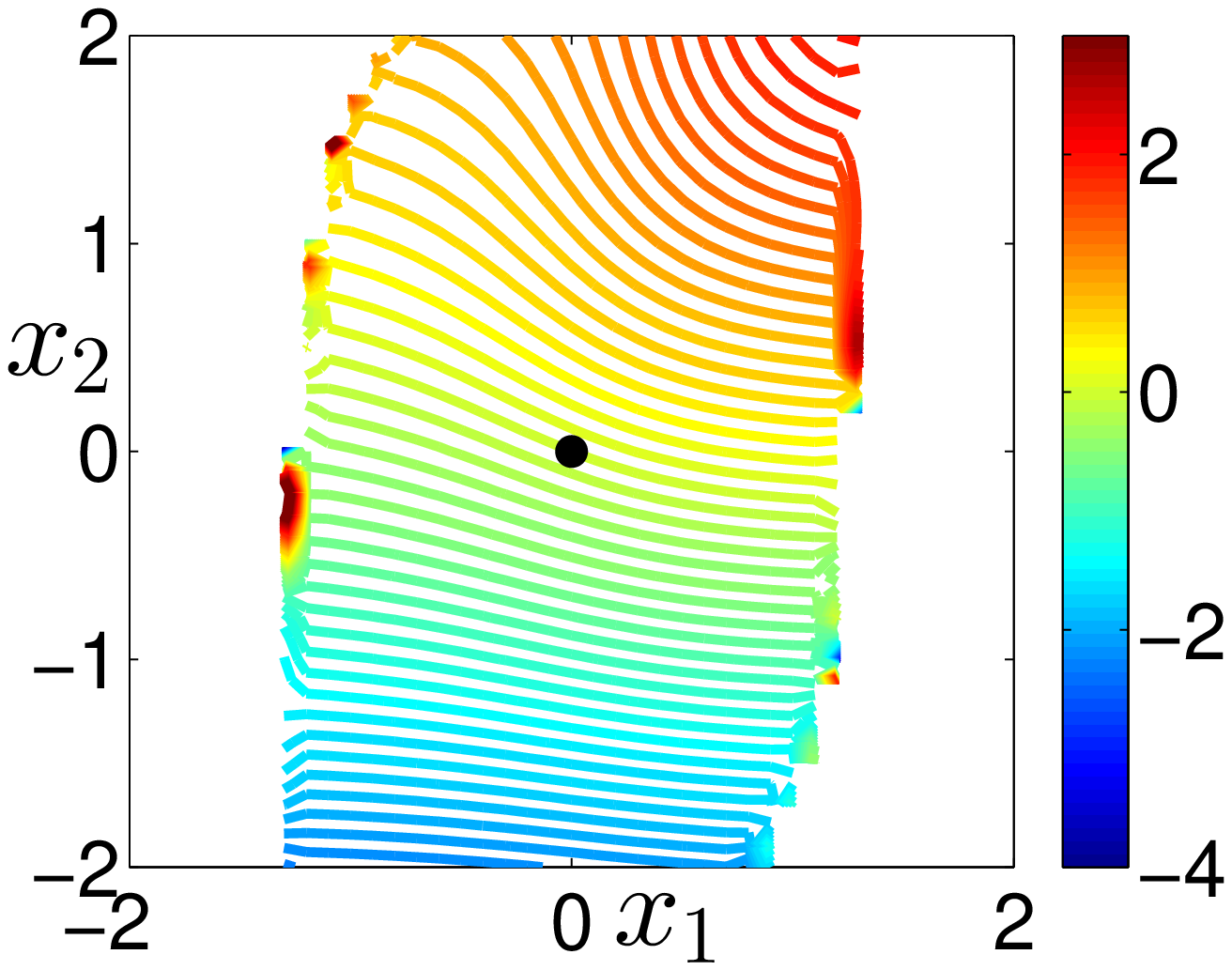}}
   \caption{In Example \ref{ex1}, the eigenfunctions are not analytic in the whole basin of attraction and the Taylor expansion-based method cannot compute them in the entire set $X=[-2,2]\times [-2,2]$. \emph{Left.} Level sets of the eigenfunction $\phi_{\lambda_1^*}$ ($\lambda_1\approx -0.698$). \emph{Right.} Level sets of the eigenfunction $\phi_{\lambda_2^*}$ ($\lambda_2 \approx -1.052$). (The eigenfunctions are computed with a Taylor expansion to the $75th$ order.)}
   \label{eigenfct_ex1}
\end{figure}

\subsection{Bernstein polynomial-based method for the fixed point}
\label{fx_pt_Bernstein}

Since the eigenfunctions are always continuous (but maybe not analytic) on the region of attraction, we can approximate them with polynomials, according to Weierstrass theorem. Also, the bad convergence results obtained with Taylor polynomials when the eigenfunctions are not analytic (see Example \ref{ex1}) can be drastically improved by considering Bernstein polynomials \cite{Bernstein_poly}. For $x\in[0,1]^N$, the eigenfunctions can be expanded in the basis of (multivariate) Bernstein polynomials of degree $s$ (in each variable)
\begin{equation*}
B_{k_1,\dots,k_N}^{s}(x) = \prod_{i=1}^N \binom{s}{k_i} \, x_i^{k_i} \, (1-x_i)^{s-k_i}\,, \quad 0\leq k_i \leq s \quad \forall i
\end{equation*}
and we have the approximation
\begin{equation}
\label{Phi_vector_Bernstein}
\phi_\lambda(x) \approx {\Phi^{(s)}}^T \, B^{s}(x)
\end{equation}
\blue{where $B^{s}(x)$ is the $(s+1)^N$-dimensional vector containing all the Bernstein polynomials $\Phi^{(s)}$ of degree $s$ (in each variable) and $\Phi^{(s)}$is the vector containing the coefficents of the expansion of $\phi_\lambda(x)$ in the basis of Bernstein polynomials.}

It is important to note that an affine change of variables might be used to ensure that $x^*\in [0,1]^N$ and that the set $[0,1]^N$ covers the whole region of interest. This change of variables modifies the dynamics to be considered in the eigenvalue equation \eqref{eig_equa}, but does not modify the eigenvalues of the Koopman operator. For instance, the change of variable $ x'=x/\alpha$ yields the dynamics $\dot{x}'=F(\alpha x')/\alpha$.

Using basic operations on Bernstein polynomials (see Appendix \ref{appendix2}), we can rewrite the eigenvalue equation \eqref{eig_equa} as
\begin{equation}
\label{Bernstein_equa1}
\left(\sum_{l=1}^N \bar{M}^l \, \bar{D}^{s,l} \, \Phi^{(s)} \right)^T \, B^{s+s'}(x) \approx \lambda \, \left(\bar{T}^{s,s'}\, \Phi^{(s)}\right)^T \, B^{s+s'}(x)
\end{equation}
where
\begin{itemize}
\item the differentiation matrices $\bar{D}^{s,l}$ are given by \eqref{coeff_D_ndim};
\item \blue{the multiplication matrices $\bar{M}^l$ are given by \eqref{M_ndim} (where $q^{(k_1,\dots,k_n)}$ are the coefficients of the expansion of $F_l(x)$ in the basis of Bernstein polynomials $B_{k_1,\dots,k_N}^{s'}(x)$);}
\item the matrix for degree raising $\bar{T}^{s,s'}$ are given by \eqref{coeff_T_ndim}.
\end{itemize}
For the computation of the eigenfunction $\phi_{\lambda_i}$, we need to impose the additional properties (i) $\phi_{\lambda_i}(x^*)=0$ since $\phi_{\lambda_i}\in \mathcal{F}_{A_c}$ and (ii) $\nabla \phi_{\lambda_i}(x^*)=w_i$, where $w_i$ is the left eigenvector (associated with the eigenvalue $\lambda_i$) of the Jacobian matrix $J$ at the fixed point (see \eqref{gradient_eig}). We have
\begin{eqnarray}
\label{Bernstein_equa2}
\left(B^{s}(x^*)\right)^T \Phi^{(s)} \,  & \approx & 0 \\
\label{Bernstein_equa3}
\left(\nabla B^{s}(x^*)\right)^T \Phi^{(s)} \,  & \approx & w_i
\end{eqnarray}
with the $(s+1)^N \times N$ matrix $\nabla B=[\partial B/\partial x_1 \cdots \partial B/\partial x_N]$. It follows from \eqref{Bernstein_equa1} (satisfied for all $x \in [0,1]^N$), \eqref{Bernstein_equa2} and \eqref{Bernstein_equa3} that $\Phi^{(s)}$ is solution of
\blue{
\begin{equation*}
\left[ \begin{array}{c}
\sum_{l=1}^N \bar{M}^l \, \bar{D}^{s,l} - \lambda_i \, \bar{T}^{s,s'} \\
\left(B^{s}(x^*)\right)^T \\
\left(\nabla B^{s}(x^*)\right)^T
\end{array} \right] \, \Phi^{(s)} = \left[ \begin{array}{c}
0\\
\vdots \\
0\\
w_i
\end{array} \right]\,.
\end{equation*}
}
This system of equations is overdetermined (since the eigenfunction is approximated with a finite number of Bernstein polynomials). Its least squares solution (obtained with the Moore-Penrose pseudoinverse) provides the coefficients $\Phi^{(s)}$, and an approximation of the eigenfunction is obtained with \eqref{Phi_vector_Bernstein}. If the system is globally stable, there exists a $C^1$ eigenfunction in $X$ and the least squares error tends to zero as $s \rightarrow \infty$.

The Koopman eigenfunctions computed with the Bernstein polynomial-based method satisfy the conditions of Proposition \ref{prop_stab_fxpt}. The existence of a (accurate) solution to the numerical method proves global stability of the system in $X$. As a confirmation of the numerical result, one can compute a Lyapunov function \eqref{Lyap_fct} with the polynomial approximations of the eigenfunctions and verify that it is decreasing along the trajectories. The following example shows that the method works with non-analytic eigenfunctions.

\begin{example}
\label{ex_Bernstein}
Consider again the system given in Example \ref{ex1}. In contrast to the Taylor-expansion based method, the Bernstein polynomial-based method can compute the non-analytic eigenfunctions in the entire set $X=[-2,2]\times [-2,2]$ (Figure \ref{eigenfct_Bernstein}), thereby proving global stability in $X$. The computation of a good candidate Lyapunov function of the form \eqref{Lyap_fct} (not shown) confirms the result.
\end{example}

\begin{figure}[h]
   \centering
   \subfigure{\includegraphics[width=7cm]{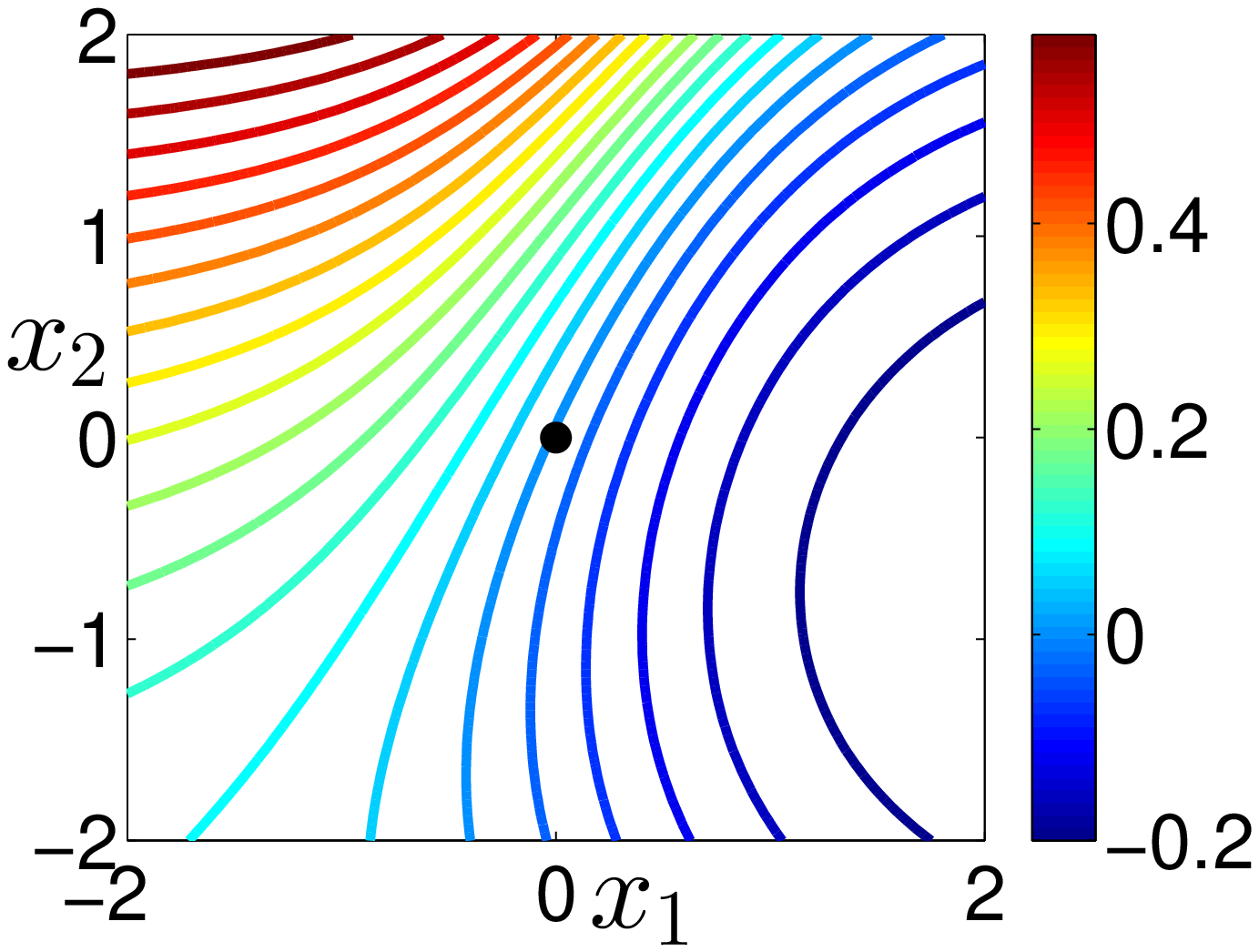}}
   \subfigure{\includegraphics[width=7cm]{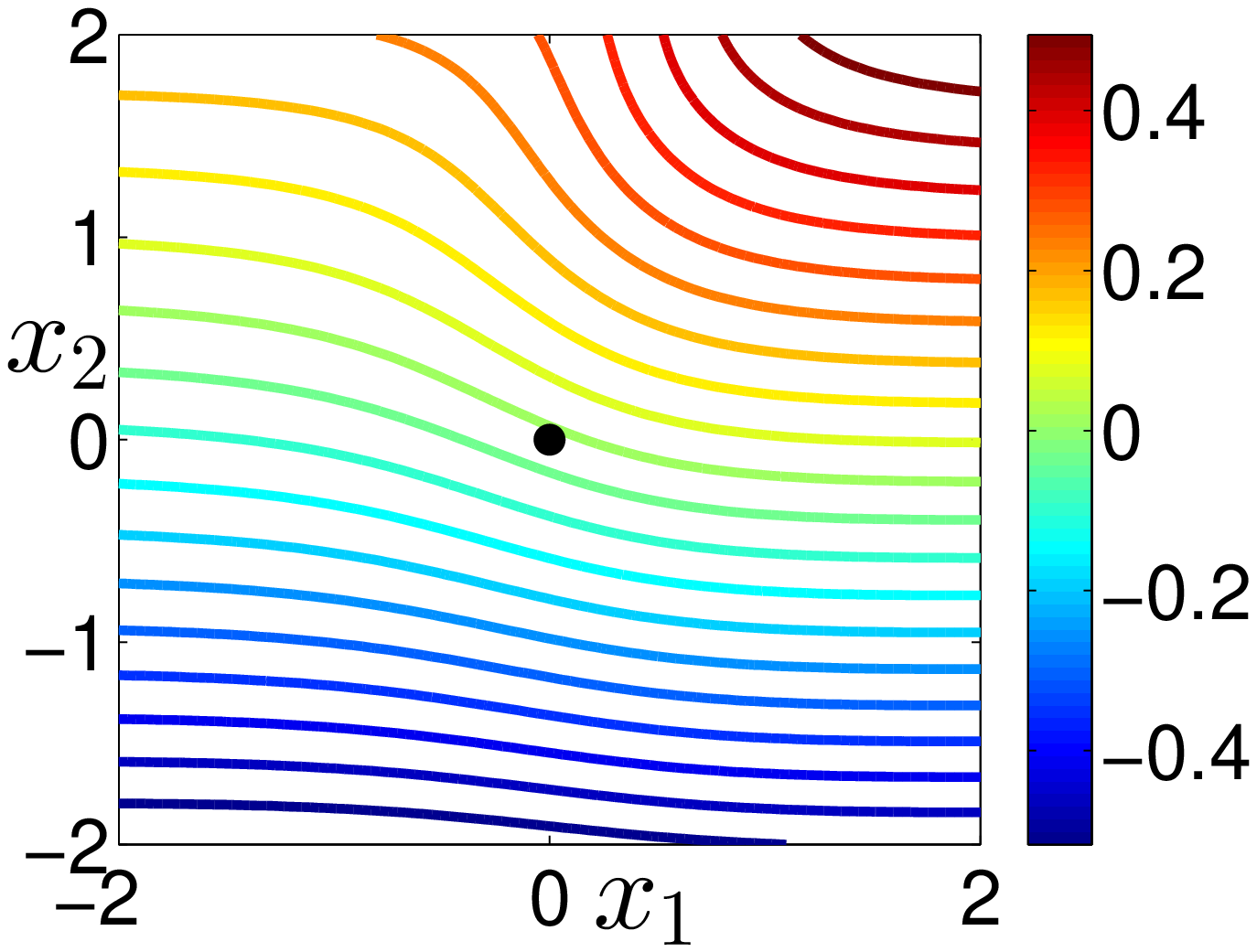}}
   \caption{Same eigenfunctions as in Figure \ref{eigenfct_ex1} (Example \ref{ex1}), but computed with the Bernstein polynomial-based method. The existence of smooth eigenfunctions in $X=[-2,2]\times [-2,2]$ implies global stability of the fixed point (black dot) on $X$. \emph{Left.} Level sets of the eigenfunction $\phi_{\lambda_1^*}$. \emph{Right.} Level sets of the eigenfunction $\phi_{\lambda_2^*}$. (The eigenfunctions are computed with Bernstein polynomials of degree $s=75$ in each variable.)}
   \label{eigenfct_Bernstein}
\end{figure}

\paragraph*{Discussion on the two methods} 

The main advantage of the Bernstein polynomial-based method is that it is characterized by good convergence properties even when the eigenfunctions are not analytic (but continuous) in $X$. This contrasts with the Taylor expansion-based method, which may fail with systems that are globally stable in $X$ (see Example \ref{ex1}). However, if the considered set $X$ contains an unstable fixed point or the boundary of the basin of attraction, the Bernstein polynomial-based method produces inaccurate results (independently of the degree of the polynomials) \blue{and cannot even provide an estimation of the basin of attraction. In other words, the method yields only a dichotomous output: globally stable in the entire set $X$ or not.}

The results obtained with the Taylor expansion-based method are accurate in a neighborhood of the equilibrium, so that global stability can always be proven in some region close to the equilibrium (i.e. inside a closed level set of a candidate Lyapunov function of the form \eqref{Lyap_fct}). The method is therefore useful to obtain successive approximations of the region of stability: Taylor polynomials of higher degree are gradually added to the truncated basis and the computation of higher order Taylor coefficients relies on the results obtained for lower orders in previous computations. \blue{In contrast, for the Bernstein polynomial-based method, inaccurate results cannot be used even in the neighborhood of the equilibrium. The degree of the polynomials must be increased and a different basis of Bernstein polynomials must be considered.}

\blue{The two methods require only a short analytic pre-processing (i.e. expansion of the vector field in the chosen polynomial basis) and compute the Koopman eigenfunctions of planar systems in a few minutes.}

\subsection{Bernstein polynomial-based method for the limit cycle}
\label{sub_sub_lim_cycle}

When the attractor is a limit cycle, the eigenfunctions can be computed on a polynomial basis in the direction transverse to the attractor and on a Fourier basis in the direction tangential to the attractor. (Recall that, according to \cite{Mohr}, an appropriate space of functions is the space of polynomials with coefficients corresponding to periodic functions on the attractor.) Since the presence of an unstable fixed point implies that the eigenfunctions are analytic only in a (usually close) neighborhood of the limit cycle, it is necessary to use Bernstein polynomials rather than monomials.

We assume that the dynamics are two-dimensional ($N=2$) and that the origin is contained inside the limit cycle (possibly after a translation of the dynamics). We consider polar-type coordinates $(\theta,y)$ such that $x=x^\gamma(\theta) + (y+\Delta) e_r(\theta)$, where $x^\gamma(\theta)$ is a point of the limit cycle, $e_r(\theta)$ is a (not necessarily unit) radial vector aligned with $x^\gamma(\theta)$, and $\Delta$ is a constant. We suppose here that the limit cycle---i.e. the set $\{(\theta,y)|\theta=[0,2\pi),y=-\Delta\}$---can be parametrized by the variable $\theta$. In the new coordinates, we have the dynamics $\dot{\theta}  =  F_\theta(\theta,y)$, $\dot{y}  =  F_y(\theta,y)$ (see \eqref{F_theta}-\eqref{F_y} in Appendix \ref{appendix3}) and the eigenvalue equation \eqref{eig_equa} is given by
\begin{equation}
\label{eigenvalue_eq_polar}
F_y \, \frac{\partial \phi_\lambda}{\partial y} + F_\theta \, \frac{\partial \phi_\lambda}{\partial \theta}=\lambda \phi_\lambda\,.
\end{equation}
In the annular region $y\in[0,1]$, we can expand the eigenfunctions in a Bernstein polynomials basis of degree $s$ ($y$ coordinate) and in a Fourier basis ($\theta$ coordinate). Considering a truncated Fourier expansion $|n| \leq \bar{n}$, \blue{the functions of the Fourier-Bernstein basis are the components of the $(2\bar{n}+1)(s+1)$-dimensional vector
\begin{equation}
\label{Fourier_Bernstein_basis}
B^{\bar{n},s} (\theta,y)=\left[\begin{array}{c}e^{-i\bar{n} \theta} \\
e^{i(-\bar{n}+1)\theta} \\
\vdots \\
e^{i\bar{n} \theta}
\end{array} \right] \otimes \left[\begin{array}{c} b_1^{s}(y) \\
\vdots \\
b_s^{s+1}(y)
\end{array} \right] 
\end{equation}
with $b^{s}_{k+1}(y)=\binom{s}{k} \, y^{k} \, (1-y)^{s-k}$ and we have the expansion
\begin{equation}
\label{Phi_vector_Bernstein2}
\phi_\lambda(x) \approx {\Phi^{(\bar{n},s)}}^T \, B^{\bar{n},s}(\theta,y) \,.
\end{equation}
}
Using basic operations on Fourier series and Bernstein polynomials (see Appendix \ref{appendix1}), we can rewrite \eqref{eigenvalue_eq_polar} as
\begin{equation}
\label{Bernstein_lim_cyc_equa1}
\left( \left(\bar{M}^y \bar{D}^y + \bar{M}^\theta \bar{D}^\theta \right) \Phi^{(\bar{n},s)} \right)^T \, B^{\bar{n},s}(\theta,y) \approx \lambda \, \left(\bar{T} \Phi^{(\bar{n},s)} \right)^T \, B^{\bar{n},s}(\theta,y)
\end{equation}
with
\begin{itemize}
\item the multiplication matrices
\begin{equation*}
\bar{M}^l=\sum_{n=-\bar{n}}^{\bar{n}} \sum_{k=0}^{s'} F_l^{(n,k)} M^{n,\theta} \otimes M^{k,y}
\end{equation*}
where $M^{n,\theta}$ are the $(2\bar{n}+1) \times (2\bar{n}+1)$ matrices with entries $M^{n,\theta}_{ij}=1$ if $j=i-n$ and $M^{n,\theta}_{ij}=0$ otherwise, $M^{k,y}$ are the $(s+s'+1) \times (s+1)$ matrices given by \eqref{coeff_M_ndim} (with $k_l=k$), \blue{and $F_l^{(n,k)}$ are the coefficients of $F_l(x)$ in the Fourier-Bernstein basis $e^{in\theta} b_{k+1}^{s'}$, $n\in[-\bar{n},n]$, $k\in[0,s']$;}
\item the differentiation matrices
\begin{eqnarray*}
\bar{D}^y & = & I^{2\bar{n}+1} \otimes D^s \\
 \bar{D}^\theta & = & \textrm{diag}(-i\bar{n}, i(-\bar{n}+1), \dots, i\bar{n}) \otimes I^{s+1}
\end{eqnarray*}
where $\bar{D}^s$ is the $(s+1) \times (s+1)$ matrix given by \eqref{coeff_D} and with $I^{2\bar{n}+1}$ and $I^{s+1}$ the $(2\bar{n}+1) \times (2\bar{n}+1)$ and $(s+1) \times (s+1)$ identity matrices;
\item the matrix for degree raising
\begin{equation*}
\bar{T} = I^{2\bar{n}+1} \otimes T^{s,s'}
\end{equation*}
where $T^{s,s'}$ is the $(s+s'+1) \times (s+1)$ matrix given by \eqref{coeff_T}.
\end{itemize}

\blue{
For the computation of the eigenfunction $\phi_{\lambda}$ where $\lambda$ is the nonzero Floquet exponent of the limit cycle, we need to impose the value of $\phi_{\lambda}$ and $\partial \phi_{\lambda}/\partial y$ on the limit cycle, i.e. for $y=-\Delta$.
Let $c_1$ and $c_2$ denote the $(2\bar{n}+1)$-dimensional vectors such that $\phi_{\lambda}(\theta,-\Delta) \approx [e^{-i \bar{n} \theta} \cdots e^{i \bar{n} \theta}] c_1$ and $\frac{\partial \phi_{\lambda}}{\partial y}(\theta,-\Delta) \approx [e^{-i \bar{n} \theta} \cdots e^{i \bar{n} \theta}] c_2$. Since $\phi_{\lambda}(\theta,-\Delta)=0$ $\forall \theta$, we have $c_1=0$. Also, it can be shown from \eqref{eigenvalue_eq_polar} that
\begin{equation*}
\frac{\partial \phi_{\lambda}}{\partial y}(\theta,-\Delta) = \exp \left(\int_0^\theta \left(\lambda-\frac{\partial F_y}{\partial y}(\sigma,-\Delta)\right)\Big/F_\theta(\sigma,-\Delta) \, d\sigma \right) \,,
\end{equation*}
which allows us to obtain the Fourier coefficients $c_2$. (Note that $\partial F_y/\partial y$ can be computed from \eqref{F_theta}-\eqref{F_y}.) We have
\begin{eqnarray}
\label{Bernstein_lim_cyc_equa2}
{\Phi^{(\bar{n},s)}}^T \, I^{2\bar{n}+1} \otimes b^{s}(-\Delta) & \approx & 0\\
\label{Bernstein_lim_cyc_equa3}
{\Phi^{(\bar{n},s)}}^T \, I^{2\bar{n}+1} \otimes \frac{d b^{s}}{d y}(-\Delta) & \approx & c_2^T
\end{eqnarray}
and it follows from \eqref{Bernstein_lim_cyc_equa1}, \eqref{Bernstein_lim_cyc_equa2} and \eqref{Bernstein_lim_cyc_equa3} that $\Phi^{(\bar{n},s)}$ is solution of the equation
\begin{equation*}
\left[ \begin{array}{c}
\bar{M}^y \bar{D}^{s,y} + \bar{M}^\theta \bar{D}^{s,\theta} - \lambda \bar{T}^{s,s'}   \\ 
I^{2\bar{n}+1} \otimes \left(b^{s}(-\Delta)\right)^T \\
I^{2\bar{n}+1} \otimes \left(\frac{db^{s}}{dy}(-\Delta)\right)^T
\end{array} \right]\, \Phi^{(\bar{n},s)} = \left[ \begin{array}{c} 0 \\
\vdots \\
0 \\
c_2
\end{array} \right] \,.
\end{equation*}
}
The coefficients $\Phi^{(s)}$ can be computed by least squares estimation and an approximation of the eigenfunction is obtained with \eqref{Phi_vector_Bernstein2}.\\

\begin{remark}[Symmetry]
\label{rem_symmetry}
The eigenfunctions are characterized by the same symmetry properties as the dynamics, in which case the coefficients $\phi_\lambda^{(m n,k)}$, with some positive integer $m$, are the only nonzero coefficients for all $n,k$. This can be exploited to increase the computational efficiency of the numerical scheme.
\end{remark}

The Koopman eigenfunction computed through the numerical method satisfies the conditions of Proposition \ref{prop_stab_limcyc}. The existence of a (accurate) solution to the numerical method proves the global stability of the limit cycle in $X$. This is illustrated in the following examples. \blue{Depending on the size of the basis, the method takes from one minute (Example \ref{ex_polar}) to about one hour (Example \ref{ex_vdp_lim_cycle}).}
\begin{example}
\label{ex_polar}
The dynamics in polar coordinates
\begin{eqnarray*}
\dot{\theta} & = & 1 \\
\dot{r} & = &  \left(2+\cos 6\theta -\cos 10 \theta \right) r(1-r^2)
\end{eqnarray*}
admit a limit cycle at $r=1$. The Koopman eigenfunction associated with the nonzero Lyapunov exponent ($\lambda=-4$) is computed on the annular region $(\theta,r)\in [0,2\pi) \times [1,3]$ (Figure \ref{lim_cycle_Bernstein}(a)). According to Proposition \ref{prop_stab_limcyc}, the limit cycle is stable on that region, a property which is verified by the fact that $\dot{r}\leq 0 $ for all $r \geq 1$.
\end{example}
\begin{example}
\label{ex_vdp_lim_cycle}
We consider the Van der Pol system
\begin{eqnarray*}
\dot{x}_1 & = & x_2 \\
\dot{x}_2 & = & -x_1+x_2-x_1^2 x_2
\end{eqnarray*}
(i.e. dynamics \eqref{vdp_rev1}-\eqref{vdp_rev2} in reverse time) which admits a stable limit cycle. Since a Koopman eigenfunction satisfying the assumptions of Proposition \ref{prop_stab_limcyc} can be computed on an annular region around the limit cycle (Figure \ref{lim_cycle_Bernstein}(b)), the limit cycle is globally stable on that region.
\end{example}

\begin{figure}[h]
   \centering
   \subfigure[Example \ref{ex_polar}]{\includegraphics[width=7cm]{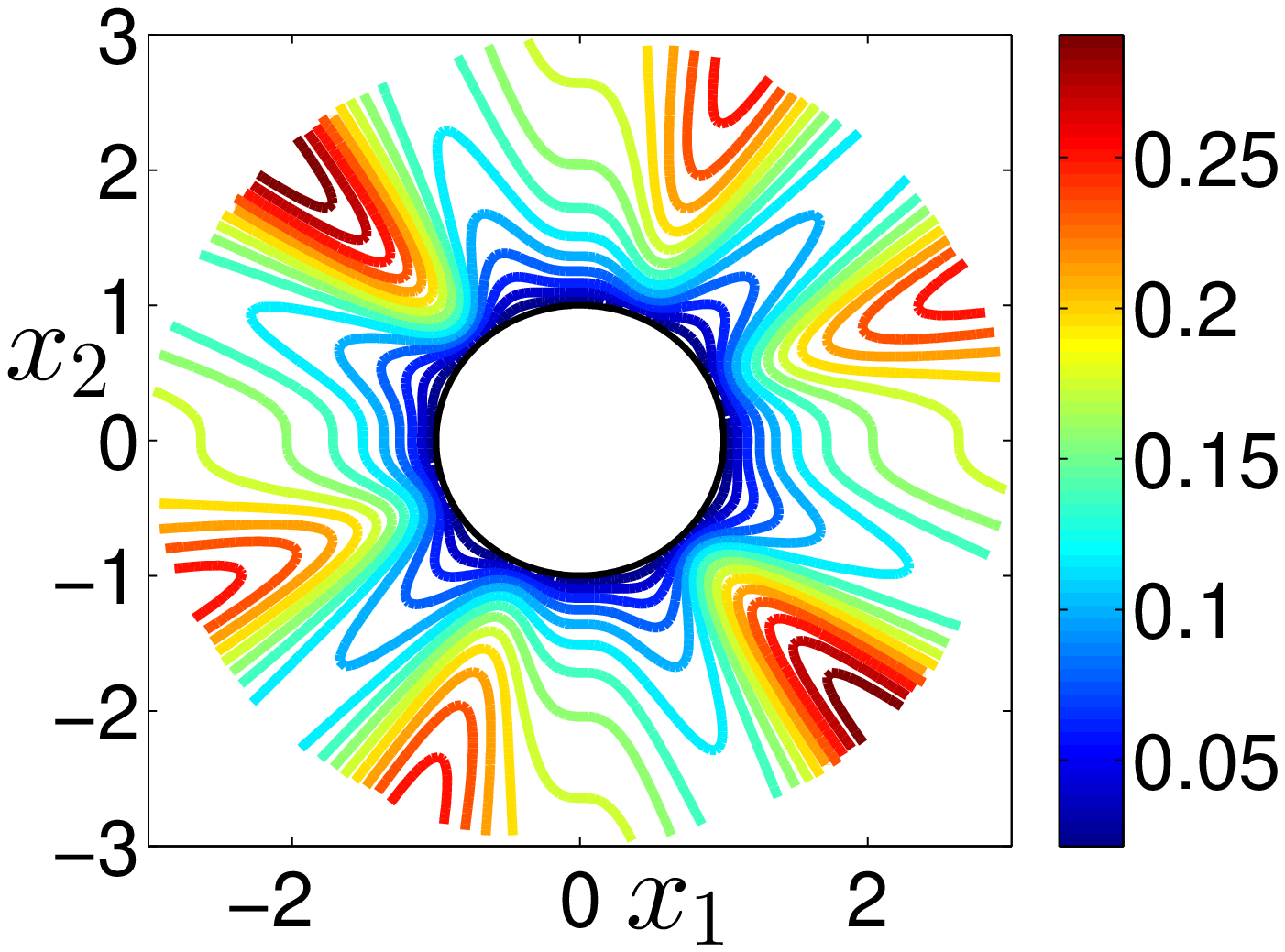}}
   \subfigure[Example \ref{ex_vdp_lim_cycle}]{\includegraphics[width=7cm]{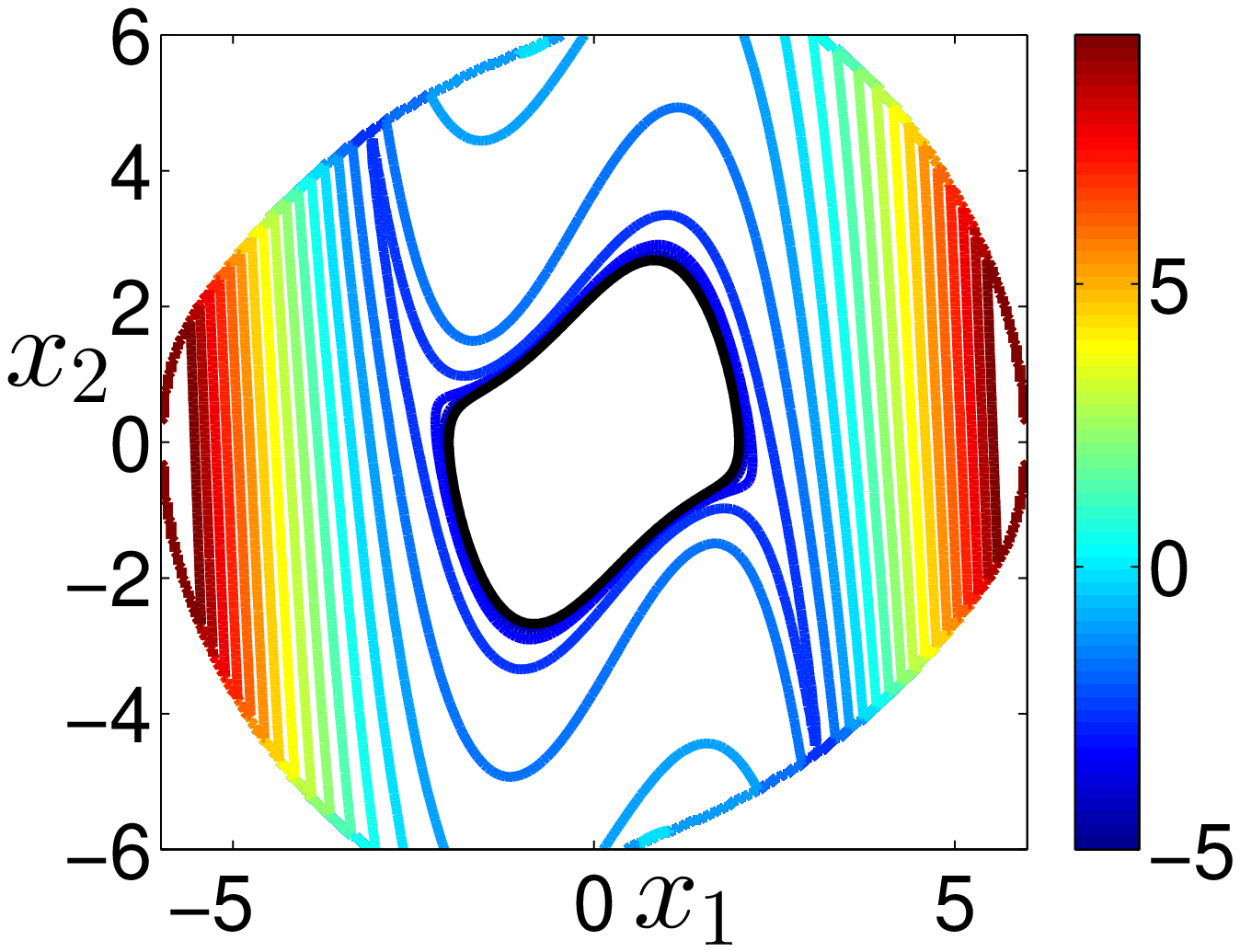}}
   \caption{An eigenfunction satisfying the assumptions of Proposition \ref{prop_stab_limcyc} is computed in an annular region around the limit cycle (black curve) and implies global stability of the limit cycle in that region. (a) The eigenfunction is computed with Bernstein polynomials of degree $s=20$ and Fourier series truncated at $\bar{n}=40$. Other parameters: $\Delta=0$, $\|e_r\|=2$. The figure shows the level sets of $\log |\phi_\lambda|$. (b) The eigenfunction is computed with Bernstein polynomials of degree $s=30$ and Fourier series truncated at $\bar{n}=200$. Other parameters: $\Delta=0$, $\|e_r\|=4$. In the two examples, $\phi_\lambda^{(2 n+1,k)}=0$ for all $n,k$ due to symmetry properties, so that only even-numbered harmonics are considered (see Remark \ref{rem_symmetry}).}
   \label{lim_cycle_Bernstein}
\end{figure}

\section{Conclusion}
\label{conclu}

Using an operator-theoretic framework, we have obtained global stability results for nonlinear systems. These results can be interpreted as the global equivalent to classic results for local stability. In particular, necessary and sufficient conditions have been derived, which rely on the existence of continuously differentiable eigenfunctions of the Koopman operator. The theoretical results are complemented with several numerical methods which are based on the decomposition onto a particular polynomial basis. A first method using monomials (i.e. Taylor expansion) enables us to approximate the region of stability of a fixed point. A second method using Bernstein polynomials is well-suited to the case of non-analytic eigenfunctions and enables us to prove (or disprove) the global stability of a fixed point or a limit cycle in a given region of interest.

The results presented in this paper might pave the way for an alternative approach to global stability analysis. But since the eigenfunctions capture more properties than mere stability (e.g. rate of convergence, etc.), the Koopman operator-theoretic framework is also a general method for studying the global behavior of dissipative nonlinear systems.

We envision several future research directions with some challenges to be overcome. Other numerical methods computationally cheaper than the methods proposed in this paper could be proposed (in particular for the limit cycle). Considering SOS techniques \cite{Parrilo_thesis} and the dual method of moments \cite{Lasserre_book} might be a promising approach toward this aim (see also \cite{Henrion2} in the context of region of attraction estimation). In addition, due to the curse of dimensionality, the numerical methods based on polynomial approximations are typically limited to \blue{low-dimensional systems}. For higher dimensions, the Koopman eigenfunctions can be obtained through averages methods as long as the computation of trajectories is permitted \cite{MMM_isostables,Mezic}. \blue{Otherwise, novel methods could be developed, for instance inspired from compressed sensing (combining polynomial expansions in low-dimensional subspaces with averages methods using partial knowledge of the trajectories), sampling (e.g. extended dynamic mode decomposition \cite{Rowley_EDMD}), or machine learning (e.g. kernel-based method \cite{Williams_kernel}). Since the operator-theoretic approach is general, it could also be further exploited to derive theoretical stability results and associated numerical methods for quasiperiodic tori and chaotic attractors. In the same context, one could consider the case of non hyperbolic attractors in detail and establish connections to center manifold theory. Appropriate numerical methods should also be developed in this case. We also believe that the theoretical stability results could be directly extended to discrete-time maps, but the associated numerical methods might be different since they should rely on an (algebraic) eigenvalue equation.} Finally, using the operator-theoretic framework in the context of input-output systems is still an open problem.

\appendix

\subsection{Operations on one-dimensional Bernstein polynomials (see also \cite{Bernstein_poly})}
\label{appendix1}

Suppose that $x\in \mathbb{R}$ and denote by $b^{s}(x)$ the $(s+1)$-dimensional vector of Bernstein polynomials of degree $s$
\begin{equation*}
b^{s}_{j+1}(x)=\binom{s}{j} \, x^{j} \, (1-x)^{s-j} \qquad j=0,\dots,s\,.
\end{equation*}
Assume that the polynomial $p(x)$ can be expressed as the product $p(x)=P^T \, b^{s}(x)$, where $P$ is a $(s+1)$-dimensional vector. For the polynomial $q(x)$, we have similarly $q(x)=Q^T \, b^{s'}(x)$, where $Q$ is a $(s'+1)$-dimensional vector.

\paragraph*{Differentiation} The differentiation of $p$ satisfies
\begin{equation*}
\frac{d p}{dx} = \left(D^s \, P \right)^T \, b^{s}(x)\,, 
\end{equation*}
where $D^s$ is a $(s+1) \times (s+1)$ matrix with the entries
\begin{equation}
\label{coeff_D}
D^s_{ij}=\begin{cases}
s-i+1 & \textrm{if } i=j-1 \\
-s+2(i-1) & \textrm{if } i=j \\
-i+1 & \textrm{if } i=j+1 \\
0 & \textrm{otherwise}
\end{cases}\,.
\end{equation}

\paragraph*{Multiplication} The multiplication of $p$ by $q$ satisfies
\begin{equation*}
q(x)\,p(x)= \left(M\, P \right)^T \, b^{s+s'}(x) 
\end{equation*}
where $M$ is a $(s+s'+1) \times (s+1)$ matrix with the entries
\begin{equation}
\label{coeff_M}
M_{ij}=\begin{cases}
Q_{i-j+1}\, \binom{s}{j-1}  \binom{s'}{i-j}/\binom{s+s'}{i-1} & \textrm{if } j\in[\max(1,i-s'),\min(s+1,i)] \\
0 & \textrm{otherwise}
\end{cases}\,.
\end{equation}

\paragraph*{Degree raising} Consider an integer $r>0$. We have
\begin{equation*}
p(x)=\left(T^{s,r}\,P \right)^T \, b^{s+r}(x)
\end{equation*}
where $T^{s,r}$ is a $(s+r+1)\times (s+1)$ matrix with the entries
\begin{equation}
\label{coeff_T}
T^{s,r}_{ij}=\begin{cases}
\binom{s}{j-1}  \binom{r}{i-j}/\binom{s+r}{i-1} & \textrm{if } i\in[j,j+r] \\
0 & \textrm{otherwise}
\end{cases}\,.
\end{equation}

\subsection{Operations on multivariate Bernstein polynomials}
\label{appendix2}

Suppose that $x\in \mathbb{R}^N$ and denote by $B^{s}(x)$ the vector of Bernstein polynomials (of degree $s$ in each variable $x_i$, $i=1,\dots,N$), i.e. $B^{s}(x)=b^{s}(x_1) \otimes \cdots \otimes b^{s}(x_N)$. Assume that the multivariate polynomial $p(x)$ can be expressed as the product $p(x)=P^T \, B^{s}(x)$, where $P$ is a vector. Similarly, we have
\begin{equation}
\label{product_q}
q(x)=Q^T \, B^{s'}(x)\,.
\end{equation} 

\paragraph*{Differentiation} The partial derivatives of $p$ are given by
\begin{equation}
\label{coeff_D_ndim}
\frac{\partial p}{\partial x_l} = \left(\bar{D}^{s,l} \, P \right)^T \, B^{s}(x) \,, 
\end{equation}
with
\begin{equation*}
\bar{D}^{s,l}= \overbrace{I^{s+1} \otimes \cdots \otimes I^{s+1}}^{l-1 \textrm{ times}} \otimes D^{s} \otimes \overbrace{I^{s+1} \otimes \cdots \otimes I^{s+1}}^{N-l \textrm{ times}} \,,
\end{equation*}
where $I^{s+1}$ is the $(s+1) \times (s+1)$ identity matrix and $D^{s}$ is given by \eqref{coeff_D}.

\paragraph*{Multiplication} The multiplication of $p$ by $q$ satisfies
\begin{equation*}
q(x)\,p(x)= \left(\bar{M} \, P \right)^T \, B^{s+s'}(x)
\end{equation*}
with
\begin{equation}
\label{M_ndim}
\bar{M}=\sum_{k_1=0}^{s'} \cdots \sum_{k_n=0}^{s'} q^{(k_1,\dots,k_n)} M^{k_1} \otimes \cdots \otimes M^{k_n}
\end{equation}
where $q^{(k_1,\dots,k_n)}$ is the component of $Q$ which multiplies $b^{s'}_{k_1+1}(x_1)\dots b^{s'}_{k_n+1}(x_n)$ in \eqref{product_q} and where $M^{k_l}$ is a $(s+s'+1) \times (s+1)$ matrix with entries
\begin{equation}
\label{coeff_M_ndim}
M^{k_l}_{ij}=\begin{cases}
\binom{s}{j-1}  \binom{s'}{k_l}/\binom{s+s'}{i-1} & \textrm{if } j=i-k_l \\
0 & \textrm{otherwise}
\end{cases}\,.
\end{equation}
This corresponds to the multiplication matrix \eqref{coeff_M} with $Q_i=1$ for $i=k_l+1$ and $Q_i=0$ otherwise (i.e. multiplication by the Bernstein polynomial $q(x)=b^s_{k_l+1}(x)$).

\paragraph*{Degree raising} Consider the integer $r>0$. We have
\begin{equation*}
p(x)=\left(\bar{T}^{s,r} \,P \right)^T \, B^{s+r,\dots,s+r}(x)
\end{equation*}
with 
\begin{equation}
\label{coeff_T_ndim}
\bar{T}^{s,r} =  \overbrace{T^{s,r} \otimes \cdots \otimes T^{s,r}}^{n \textrm{ times}}\,,
\end{equation}
where $T^{s,r}$ is given by \eqref{coeff_T}.

\subsection{Dynamics in polar-type coordinates}
\label{appendix3}

Considering the relationship $x=x^\gamma(\theta) + (y + \Delta) e_r(\theta)$, we have
\begin{equation*}
F(x)=\dot{x}= \frac{d x^\gamma}{d\theta} \, \dot{\theta} + e_r \, \dot{y} + (y+\Delta) \, e_\theta \, \dot{\theta} \,,
\end{equation*}
where $e_\theta=d e_r/d\theta$. Since $e_r \cdot e_\theta=0$ and $\|e_\theta\|=\|e_r\|$, it follows that
\begin{eqnarray*}
F(x) \cdot e_\theta & = & \left(\frac{dx^\gamma}{d\theta} \cdot e_\theta \right) \dot{\theta} + (y+\Delta) \, \|e_r\|^2 \, \dot{\theta} \,. \\
F(x) \cdot e_r & = & \left(\frac{dx^\gamma}{d\theta} \cdot e_r \right) \dot{\theta} + \|e_r\|^2 \, \dot{y}\,.
\end{eqnarray*}
Equivalently, we have
\begin{equation}
\label{F_theta}
\dot{\theta} = \frac{F(x) \cdot \left(e_\theta/\|e_r\|\right)}{\|x_\gamma(\theta)\|+(y+\Delta)\, \|e_r\|} \triangleq F_\theta(\theta,y)
\end{equation}
where we used $\frac{dx^\gamma}{d\theta} = \frac{d(\|x^\gamma\| e_r/\|e_r\|)}{d\theta} = \frac{d\|x^\gamma\|}{d\theta} \, e_r/\|e_r\| + \|x^\gamma\|\, e_\theta/\|e_r\|$, and we have
\begin{equation}
\label{F_y}
\dot{y} = \left( F(x)- F_\theta(\theta,y) \frac{d x^\gamma}{d\theta} \right) \cdot \left(e_r/\|e_r\|^2\right)\triangleq F_y(\theta,y)\,.
\end{equation}

\bibliographystyle{plain}

\begin{thebibliography}{100}

\bibitem{Budisic_Koopman}
M.~Budi{\v{s}}i{\'c}, R.~Mohr, and I.~Mezi{\'c}.
\newblock {Applied Koopmanism}.
\newblock {\em Chaos}, 22(4):047510--047510, 2012.

\bibitem{Carleman}
T.~Carleman.
\newblock Application de la théorie des équations integrales lineaires aux
  systèmes d'équations différentielles nonlinéaires.
\newblock {\em Acta Mathematica}, 59:63--68, 1932.


\bibitem{Bernstein_poly}
R.~T. Farouki.
\newblock {The Bernstein polynomial basis: A centennial retrospective}.
\newblock {\em Computer Aided Geometric Design}, 29(6):379--419, 2012.

\bibitem{Forni}
F.~Forni and R.~Sepulchre.
\newblock {A differential Lyapunov framework for contraction analysis}.
\newblock {\em IEEE Transactions On Automatic Control}, 59(3):614--628, March
  2014.

\bibitem{Gaspard}
P.~Gaspard, G.~Nicolis, A.~Provata, and S.~Tasaki.
\newblock Spectral signature of the pitchfork bifurcation: Liouville equation
  approach.
\newblock {\em Physical Review E}, 51(1):74, 1995.

\bibitem{Gaspard2}
P.~Gaspard and S.~Tasaki.
\newblock {Liouvillian dynamics of the Hopf bifurcation}.
\newblock {\em Physical Review E}, 64(5):056232, 2001.

\bibitem{Khalil}
H.~K. Khalil and J.~W. Grizzle.
\newblock {\em Nonlinear systems}, volume~3.
\newblock Prentice hall New Jersey, 1996.

\bibitem{Henrion2}
M.~Korda, D.~Henrion, and C.~N. Jones.
\newblock Controller design and region of attraction estimation for nonlinear
  dynamical systems.
\newblock arXiv preprint arXiv:1310.2213, 2013.

\bibitem{Lan}
Y.~Lan and I.~Mezi{\'c}.
\newblock {Linearization in the large of nonlinear systems and Koopman operator
  spectrum}.
\newblock {\em Physica D}, 242:42--53, 2013.

\bibitem{Lasota_book}
A.~Lasota and M.~C. Mackey.
\newblock {\em Chaos, Fractals, and Noise: stochastic aspects of dynamics}.
\newblock Springer-Verlag, 1994.

\bibitem{Lasserre_book}
J.-B. Lasserre.
\newblock {\em Moments, positive polynomials and their applications}, volume~1.
\newblock World Scientific, 2009.

\bibitem{Mauroy_Mezic}
A.~Mauroy and I.~Mezi{\'c}.
\newblock {On the use of Fourier averages to compute the global isochrons of
  (quasi)periodic dynamics}.
\newblock {\em Chaos}, 22(3):033112, 2012.

\bibitem{MauroyMezic_CDC}
A.~Mauroy and I.~Mezi{\'c}.
\newblock A spectral operator-theoretic framework for global stability.
\newblock In {\em Proceedings of the 52th IEEE Conference on Decision and
  Control}, pages 5234--5239, December 2013.

\bibitem{MMM_isostables}
A.~Mauroy, I.~Mezi{\'c}, and J.~Moehlis.
\newblock {Isostables, isochrons, and Koopman spectrum for the action-angle
  representation of stable fixed point dynamics}.
\newblock {\em Physica D: Nonlinear Phenomena}, 261:19--30, October 2013.

\bibitem{Mezic}
I.~Mezi{\'c}.
\newblock Spectral properties of dynamical systems, model reduction and
  decompositions.
\newblock {\em Nonlinear Dynamics}, 41(1-3):309--325, 2005.

\bibitem{Mezic_ann_rev}
I.~Mezi{\'c}.
\newblock Analysis of fluid flows via spectral properties of {Koopman}
  operator.
\newblock {\em Annual Review of Fluid Mechanics}, 45, January 2013.

\bibitem{Mohr}
R.~Mohr and I.~Mezi{\'c}.
\newblock {Construction of eigenfunctions for scalar-type operators via Laplace
  averages with connections to the Koopman operator}.
\newblock http://arxiv.org/abs/1403.6559.

\bibitem{Parrilo_thesis}
P.~A. Parrilo.
\newblock {\em Structured semidefinite programs and semialgebraic geometry
  methods in robustness and optimization}.
\newblock PhD thesis, California Institute of Technology, 2000.

\bibitem{Umesh_density}
R.~Rajaram, U.~Vaidya, M.~Fardad, and B.~Ganapathysubramanian.
\newblock {Stability in the almost everywhere sense: A linear transfer operator
  approach}.
\newblock {\em Journal of Mathematical Analysis and Applications},
  368(1):144--156, 2010.

\bibitem{Rantzer_density}
A.~Rantzer.
\newblock {A dual to Lyapunov's stability theorem}.
\newblock {\em Systems \& Control Letters}, 42(3):161--168, 2001.

\bibitem{Rowley}
C.~W. Rowley, I.~Mezic, S.~Bagheri, P.~Schlatter, and D.~S. Henningson.
\newblock Spectral analysis of nonlinear flows.
\newblock {\em Journal of Fluid Mechanics}, 641:115--127, 2009.

\bibitem{Susuki}
Y.~Susuki and I.~Mezi{\'c}.
\newblock {Nonlinear Koopman modes and power system stability assessment
  without models}.
\newblock {\em IEEE Transactions On Power Systems}, 29(2):899--907, March 2014.

\bibitem{Umesh}
U.~Vaidya and P.~G. Mehta.
\newblock Lyapunov measure for almost everywhere stability.
\newblock {\em IEEE Transactions on Automatic Control}, 53(1):307--323, 2008.

\bibitem{Rowley_EDMD}
M.~O. Williams, I.~G. Kevrekidis, and C.~W. Rowley.
\newblock {A Data–Driven Approximation of the Koopman Operator: Extending
  Dynamic Mode Decomposition}.
\newblock {\em Journal of Nonlinear Science}, pages 1--40, June 2015.

\bibitem{Williams_kernel}
M.~O. Williams, C.~W. Rowley, and I.~G. Kevrekidis.
\newblock {A kernel approach to data-driven Koopman spectral analysis}.
\newblock http://arxiv.org/abs/1411.2260, 2014.

\end{thebibliography}

\end{document}